\documentclass[12pt]{amsart}
\usepackage{color,graphicx,array, amssymb, amscd,slashed,fancyhdr}
\newtheorem{Theorem}{Theorem}[section]
\newtheorem{Lemma}[Theorem]{Lemma}
\newtheorem{Proposition}[Theorem]{Proposition}
\newtheorem{Corollary}[Theorem]{Corollary}
\theoremstyle{definition}
\newtheorem{Definition}{Definition}
\theoremstyle{remark}
\newtheorem{Example}{Example}
\newtheorem{Remark}[Theorem]{Remark} 

\numberwithin{equation}{section}
\newcommand{\R}{\mathbb R}
\newcommand{\C}{\mathbb C}

\newcommand{\D}{\mathbb D}

\newcommand{\GL}{{\rm GL}_2 \mathbb C}

\newcommand{\ISU}{{\rm SU}_{1, 1}}

\newcommand{\id}{\operatorname{id}}

\newcommand{\isu}{\mathfrak{su}_{1, 1}}
\newcommand{\SL}{{\rm SL}_2 \mathbb C}

\newcommand{\LSL}{\Lambda {\rm SL}_2 \mathbb C_{\sigma}}
\newcommand{\LSLPM}{\Lambda^{\pm} {\rm SL}_{2} \mathbb C_{\sigma}}
\newcommand{\LSLM}{\Lambda^{-} {\rm SL}_{2} \mathbb C_{\sigma}}
\newcommand{\LSLMI}{\Lambda_*^{-} {\rm SL}_{2} \mathbb C_{\sigma}}
\newcommand{\LSLP}{\Lambda^+ {\rm SL}_{2} \mathbb C_{\sigma}}
\newcommand{\LSLPI}{\Lambda_*^+ {\rm SL}_{2} \mathbb C_{\sigma}}

\newcommand{\LISU}{(\Lambda {\rm SU}_{1, 1})_\sigma}

\newcommand{\ad}{\operatorname{Ad}}
\newcommand{\di}{\operatorname{diag}}

\newcommand{\Nil}{{\rm Nil}_3}
\renewcommand{\Re}{\operatorname {Re}}
\renewcommand{\Im}{\operatorname {Im}}
\newcommand{\sdz}{(dz)^{1/2}}
\newcommand{\sdzb}{(d\bar z)^{1/2}}
\begin{document}
\title{A loop group method for minimal surfaces in the 
 three-dimensional Heisenberg group}
 \author[J. F.~Dorfmeister]{Josef F. Dorfmeister}
 \address{Fakult\"at f\"ur Mathematik, 
 TU-M\"unchen, 
 Boltzmann str. 3,
 D-85747, 
 Garching, 
 Germany}
 \email{dorfm@ma.tum.de}
\author[J.~Inoguchi]{Jun-ichi Inoguchi}
 \address{Department of Mathematical Sciences, 
 Faculty of Science,
 Yamagata University, 
 Yamagata, 990--8560, Japan}
 \email{inoguchi@sci.kj.yamagata-u.ac.jp}
 \thanks{The second named author is partially supported by Kakenhi 
21546067, 24540063}
\author[S.-P.~Kobayashi]{Shimpei Kobayashi}
 \address{Department of Mathematics, Hokkaido University, 
 Sapporo, 060-0810, Japan}
 \email{shimpei@math.sci.hokudai.ac.jp}
 \thanks{The third named author is partially supported by Kakenhi 
23740042}
\date{\today}
\begin{abstract} 
 We characterize constant mean curvature surfaces in the 
 three-dimensional 
 Heisenberg group by a family of flat connections 
 on the trivial bundle $\D \times \GL$ over a simply connected 
 domain $\mathbb{D}$ in the complex plane. 
 In particular for minimal surfaces, 
 we give an immersion formula, the so-called Sym-formula, 
 and a generalized Weierstrass type representation via the loop group method.
\end{abstract}
\maketitle
\section*{Introduction}
 Surfaces of constant curvature or of constant mean curvature 
 in space forms (of both definite and indefinite type) have been 
 investigated since the beginning of differential geometry. 
 For more than fifteen years now a loop group technique has been 
 used to investigate these surfaces, 
 see \cite{Dorfmeister, Kob:Realforms}.
 
 During the last few years, surfaces of 
 constant mean curvature in more general three-dimensional manifolds 
 have been investigated. A natural target were the model spaces of 
 Thurston geometries, see \cite{Daniel:Homogeneous}.
 
 According to Thurston \cite{Thurston}, 
 there are eight model spaces of three-dimensional 
 geometries, Euclidean $3$-space $\mathbb{R}^3$, $3$-sphere 
 $\mathbb{S}^3$, 
 hyperbolic $3$-space $\mathbb{H}^3$, Riemannian products 
 $\mathbb{S}^{2}\times \mathbb{R}$ and 
 $\mathbb{H}^{2}\times \mathbb{R}$, the three-dimensional 
 Heisenberg group 
 $\mathrm{Nil}_3$, the universal covering 
 $\widetilde{\mathrm{SL}_{2}\mathbb{R}}$ of 
 the special linear group and 
 the space $\mathrm{Sol}_3$. 
 The geometrization conjecture posed by Thurston (and solved by 
 Perelman) states that these eight model spaces are the 
 \textit{building blocks} to construct any three-dimensional manifolds.
 The dimension of the isometry group of the model spaces is greater than $3$, 
 except in the case $\mathrm{Sol}_3$.  In particular, the space forms 
 $\mathbb{R}^3$, $\mathbb{S}^3$ and $\mathbb{H}^3$ have $6$-dimensional 
 isometry groups. 
 The model spaces with the exception of $\mathrm{Sol}_3$ and $\mathbb{H}^3$ 
 belong to the following 
 $2$-parameter family 
 $\{E(\kappa,\tau)\ \ \vert \ \kappa,\tau\in \mathbb{R}\}$ of 
 homogeneous Riemannian $3$-manifolds: Let
 $$
 E(\kappa,\tau)=(\mathcal{D}_{\kappa,\tau},ds_{\kappa,\tau}^{2}),
 $$
 where the domain $\mathcal{D}_{\kappa,\tau}$ is the whole 
 $3$-space $\mathbb{R}^3$ for $\kappa\geq 0$ and 
 $$
 \mathcal{D}_{\kappa,\tau}:=\{ 
 (x_1,x_2,x_3) \in \mathbb{R}^3 
 \ 
 \vert
 \ x_1^{2}+x_2^{2}<-4/\kappa
 \}
 $$
 for $\kappa<0$. The Riemannian metric 
 $ds_{\kappa,\tau}^2$ is given by
$$
 ds_{\kappa,\tau}^2=
 \frac{dx_1^2+dx_2^2}{\left(1+\frac{\kappa}{4}(x_1^2+x_2^2)\right)^2}+
 \left(
 dx_{3}+ \frac{\tau (x_{2}dx_{1}-x_{1}dx_{2})}{1+\frac{\kappa}{4}(x_1^2+x_2^2)}
 \right)^2.
$$ 
 This $2$-parameter family can be seen in the 
 local classification of all 
 homogeneous Riemannian metrics on $\mathbb{R}^3$ 
 due to Bianchi, \cite{Bianchi}, see also 
 Vranceanu \cite[p.~354]{V}. 
 Cartan classified transitive isometric actions 
 of $4$-dimensional Lie groups on Riemannian $3$-manifolds 
 \cite[pp.~293--306]{Cartan}. Thus the family $E(\kappa, \tau)$ with 
 $\kappa \in \mathbb R$ and $\tau \geq 0$ is referred to as the 
 \textit{Bianchi-Cartan-Vranceanu family}, \cite{BDI}.
 By what was said above the Bianchi-Cartan-Vranceanu family includes 
 all local three-dimensional homogeneous Riemannian metrics
 whose isometry groups have dimension greater than $3$ except
 constant negative curvature metrics. 
 The parameters 
 $\kappa$ and $\tau$ are called the 
 \textit{base curvature} and \textit{bundle curvature} of 
 $E(\kappa,\tau)$, respectively.
 
 The Heisenberg group $\mathrm{Nil}_3$ together with a 
 standard left-invariant metric is isometric to the homogeneous 
 Riemannian manifold
 $E(\kappa,\tau)$ with $\kappa=0$ and $\tau\not=0$. 
 Without loss of generality we can normalize $\tau=1/2$ for 
 the Heisenberg group: $\mathrm{Nil}_3=E(0,1/2)$. 
 Note that $E(0,1)$ is the
 \textit{Sasakian space form}, $\mathbb{R}^{3}(-3)$, see 
 \cite{Boyer, BG}.

 An important piece of progress of surface geometry 
 in $E(\kappa,\tau)$ was a result of Abresch and Rosenberg 
 \cite{Abresch-Rosenberg}:
 A certain quadratic differential turned out to be holomorphic 
 for all surfaces of constant mean curvature in the 
 above model spaces $E(\kappa,\tau)$.
  
 Since in the classical case of surfaces in space forms the holomorphicity
 of the (unperturbed) Hopf differential was crucial for the existence of a 
 loop group approach to the construction of those surfaces, the 
 question arose, to what extent a loop group approach would also 
 exist for the more general class of constant mean curvature 
 surfaces in Thurston geometries.
 
 All the model spaces are Riemannian homogeneous spaces.
 Minimal surfaces in Riemannian homogeneous spaces are
 regarded as conformally harmonic maps from Riemann surfaces. 
 Conformally harmonic maps of Riemann surfaces into Riemannian 
 symmetric spaces admit a zero curvature representation and 
 hence loop group methods can be applied.

 More precisely, the loop group method has two key ingredients.
 One is the zero curvature representation of harmonic maps.
 The zero curvature representation is equivalent 
 with the existence of a loop of flat connections
 and this representation enables us to use loop groups. 
 The other one is an appropriate loop group decomposition. 
 A loop group decomposition recovers the harmonic map 
 (minimal surfaces) from holomorphic potentials. The construction 
 of harmonic maps from prescribed potentials is now referred to as the 
 \textit{generalized Weierstrass type representation} for harmonic 
 maps, see \cite{DPW}.

 Every (compact)
 semi-simple Lie group $G$ equipped with 
 a bi-invariant (semi-)Riemannian metric is represented by 
 $G\times G/G$ as a  (semi-)Riemannian symmetric 
 space. Thus we can apply the loop group method to harmonic maps into $G$. 
 Harmonic maps from the two-sphere $\mathbb{S}^2$ or the 
 two-torus $\mathbb{T}^2$ 
 into compact semi-simple Lie groups 
 have been studied extensively, see \cite{BFPP, BP, Segal, Uhlenbeck}.
 The three-sphere $\mathbb{S}^3$ is identified with the 
 special unitary group $\mathrm{SU}_2$ equipped with a 
 bi-invariant Riemannian metric of constant curvature $1$.
 Thus we can study minimal surfaces in $\mathbb{S}^3$ by a loop group method.
 Note that harmonic tori in $\mathbb{S}^3$ have been 
 classified by Hitchin \cite{Hitchin} via the spectral curve method. 

 It is known that model spaces, except $\mathbb{S}^2\times \mathbb{R}$,
 can be realized as Lie groups equipped with left-invariant Riemannian metrics.
 Thus it has been expected to generalize the loop group method for 
 harmonic maps of Riemann surfaces into compact Lie groups equipped 
 with a bi-invariant Riemannian 
 metric to those for maps into more general Lie groups. 
 However the bi-invariant property is essential for the application of the 
 loop group method. 
%

 Thus to establish a generalized Weierstrass type representation for minimal 
 surfaces (or more generally CMC surfaces) in model spaces of Thurston 
 geometries, 
 another key ingredient is required.    
 Since $\Nil$ seems to be a particularly simple example of a 
 Thurston geometry and more work has been done for 
 this target space than for other ones, 
 we would like to attempt to introduce a loop group approach to 
 constant mean curvature surfaces in $\Nil$.
 
 The procedure is as follows: Consider a conformal 
 immersion $f:\mathbb{D}\to\Nil=E(0,1/2)$ of a 
 simply connected domain of the complex plane $\mathbb{C}$. 
 Then define a matrix valued function $\varPhi$ 
 by $\varPhi=f^{-1}\partial_{z}f$ 
 and expand it as $\varPhi=\sum_{k=1}^{3}\phi_{k}e_{k}$ relative to the 
 natural basis $\{e_1,e_2,e_3\}$ of the Lie algebra of $\mathrm{Nil}_3$. 
 The new key ingredient is the \textit{spin structure} of Riemann surfaces.
 Represent $(\phi_1,\phi_2,\phi_3)$ by 
 $(\phi_1, \phi_2, \phi_3) = ((\overline{\psi_2})^2 - \psi_1^2,  
 \;i((\overline{\psi_2})^2 + \psi_1^2), \;2 \psi_1 \overline{\psi_2})$ 
 in terms of \textit{spinors} $\{\psi_1,\psi_2\}$ that are unique up to a sign.

 It has been shown by Berdinsky \cite{Ber:Heisenberg},
 that the spinor field $\psi = (\psi_1, \psi_2)$ satisfies the 
 matrix system of equations $\partial_z \psi = \psi U$, 
 $\partial_{\bar z} \psi = \psi V$, where the coefficients of 
 $U$ and $V$ have a simple form in terms of 
 the \textit{mean curvature} $H$, 
 the \textit{conformal factor} $e^u$ of the metric, 
 the spin geometric \textit{support function} $h$ of the normal vector field of the immersion 
 and the \textit{Abresch-Rosenberg quadratic differential} $Q dz^2$. 
 In \cite{BT:Sur-Lie}, another quadratic differential, 
 $\tilde Adz^2$ where $\tilde{A}=Q/(2H+i)$ was introduced. 
 For $H =\mbox{const}$ there is, obviously, not much of a difference.
 However for $\Nil$ it turns out that $\tilde A$ is holomorphic if 
 and only if $f$ has constant mean curvature, \cite{BT:Sur-Lie},  but $Q$ is 
 holomorphic for all constant mean curvature surfaces and, in addition, 
 also for one non constant mean curvature surface, the so-called 
 \textit{Hopf cylinder} (Theorem \ref{thm:ARholo}). 
 A similar situation occurs for other Thurston geometries, see 
 \cite{ Fer-Mira}.

 One can show that for every conformal constant mean curvature immersion
  $f$ into $\Nil$ the Berdinsky system describes a harmonic map into a 
 symmetric space $\GL/ \di $ (Theorem \ref{thm:CMCcharact}).
 Thus the corresponding system can be constructed by 
 the loop group method, 
 that is, there exists an \textit{associated family} of surfaces 
 parametrized by a \textit{spectral parameter}.
 However, it is not clear so far, how one can make sure that 
 a solution spinor $\psi = (\psi_1, \psi_2)$ to 
 the Berdinsky system induces, 
 via $f^{-1}\partial_z f = \sum_{k=1}^3 \phi_k e_k$, with
 $\phi_k$ the ``Weierstrass type representations''
 formed with $\psi_1$ and $\psi_2$ as above, a (real !) immersion 
 into $\Nil$. 

 Unfortunately, a result of Berdinsky \cite{Ber:Note}
 shows that a naturally associated family of surfaces cannot 
 stay in $\Nil$ for all values of the spectral parameter 
 (Corollary \ref{coro:CMCnotinNil}).
 However, for the case of minimal surfaces in $\Nil$ this problem 
 does not arise. Therefore, as a first attempt to introduce a loop 
 group method for the discussion of constant mean curvature surfaces in 
 Thurston geometries, 
 we present in this paper a loop group approach to minimal 
 surfaces in $\Nil$.
%
%

 Moreover, for the case of minimal surfaces, the normal Gauss maps, 
 which are maps into hyperbolic $2$-space $\mathbb{H}^2$, are
 harmonic (Theorem \ref{thm:mincharact}) 
 and an immersion formula is obtained from the frame 
 of the normal Gauss map, the so-called 
 \textit{Sym-formula} (Theorem \ref{thm:Sym}), see 
 also \cite{Cartier}. Thus 
 the loop group method can be applied without restrictions
 to the case of minimal surfaces, that is, 
 a pair of meromorphic $1$-forms, through the loop group 
 decomposition, determines a minimal surface, the so-called 
 \textit{generalized Weierstrass type representation}.
 It is worthwhile to note that the associated family of a minimal surface 
 in $\Nil$ preserves the support but not the metric. This 
 gives a geometric characterization of the associated family
 of a minimal surface in $\Nil$ which is different from the case of 
 constant mean curvature surfaces in $\R^3$, where 
 the associated family preserves the metric (Corollary \ref{coro:associate}).

 This paper is organized as follows: 
 In Sections \ref{sc:Preliminaries}--\ref{sc:SurfaceTheory}, we give 
 basic results for harmonic maps and surfaces in $\Nil$. In Section \ref{sc:CMCinNil}, constant 
 mean curvature surfaces in $\Nil$ are characterized by a family of 
 flat connections on $\D \times \GL$. In Sections \ref{sc:minimal} and  
 \ref{sc:Sym}, we will concentrate on minimal surfaces.  
 In particular, minimal surfaces in $\Nil$ are characterized by a family 
 of flat connections on $\D \times \ISU$ and an immersion formula 
 in terms of an extended frame is given.
 In Sections \ref{sc:potential} and \ref{sc:Weierstrass}, 
 a generalized Weierstrass type 
 representation for minimal surfaces in $\Nil$ is given via 
 the loop group method, that is, a minimal surface 
 is recovered by a pair of holomorphic functions through
 the loop group decomposition. In Section \ref{sc:examples},
 several examples are given by the generalized Weierstrass type 
 representation established in this paper.

\textbf{Acknowledgements:} 
 This work was started when the first and third named authors visited 
 Tsinghua  University, 2011. We would like to express our sincere thanks
 to the Department of Mathematics of  Tsinghua  University 
 for its hospitality. 

\section{Minimal surfaces in Lie groups}
\label{sc:Preliminaries}
\subsection{}
 Let $G\subset \mathrm{GL}_{n}\mathbb{R}$ be a closed 
 subgroup of the real general linear group of degree $n$.
 Denote by $\mathfrak{g}$ the Lie algebra of $G$, that is, 
 the tangent space of $G$ at the identity. 
 We equip $\mathfrak g$ with an inner product $\langle \cdot,\cdot\rangle$
 and extend it to a left-invariant Riemannian metric 
 $ds^2=\langle \cdot,\cdot\rangle$ on $G$.

 Now let $f:M\to G$ be a smooth map of a Riemann surface $M$ into $G$. 
 Then 
$\alpha:=f^{-1}df$ satisfies the \textit{Maurer-Cartan equation}:
$$
d\alpha+\frac{1}{2}[\alpha \wedge \alpha]=0.
$$
 Take a local complex coordinate $z=x+iy$ defined on a 
 simply connected domain $\mathbb{D}\subset M$ and express $\alpha$ as 
$$
\alpha=\varPhi\>dz+\bar{\varPhi}\>d\bar{z}.
$$ 
 Here the coefficient matrices $\varPhi$ and $\bar {\varPhi}$ are computed as 
$$
\varPhi=f^{-1}f_{z},
\ \
\bar{\varPhi}=f^{-1}f_{\bar z}.
$$
 The subscripts $z$ and $\bar{z}$ denote the partial differentiations 
 $\partial_{z}=(\partial_{x}-i\partial_{y})/2$ and 
 $\partial_{\bar z}=(\partial_{x}+i\partial_{y})/2$,
 respectively. We note that 
 $\bar{\varPhi}$ is the complex conjugate of $\varPhi$, since $f$ takes 
 values in $G\subset \mathrm{GL}_{n}\mathbb{R}$.

 Denote the complex bilinear extension of $\langle\cdot,\cdot\rangle$ 
 to $\mathfrak g^{\mathbb C}$ by the same letter. 
 Then $f$ is a conformal immersion if and only if 
\begin{equation}\label{conformalimmersion}
\langle \varPhi,\varPhi\rangle=0,
\ \
\langle \varPhi,\bar{\varPhi}\rangle>0.
\end{equation}
 For a conformal immersion $f:M\to G$, the 
 induced metric (also called the first fundamental form)
 $\langle df,df\rangle$, is represented as $e^{u}dzd{\bar{z}}$. The 
 function $e^{u}:=2\langle f_{z},f_{\bar z}\rangle$ is called 
 the \textit{conformal factor} of the metric with respect to $z$.

 Next take an orthonormal basis $\{e_1,e_2,\cdots,e_\ell\}$ of the Lie 
 algebra $\mathfrak{g}$ ($\ell=\dim \mathfrak{g}$). Expand $\varPhi$ as
 $\varPhi=\phi_{1}e_{1}+\phi_{2}e_{2}+\cdots+\phi_{\ell}e_{\ell}$. Then
 we have the following fundamental fact.
\begin{Proposition}\label{prop:Integrablity}
 Let $f:M \rightarrow G\subset \mathrm{GL}_{n}\mathbb{R}$  
 be a conformal immersion with the conformal factor $e^u$. 
 Moreover, set $\varPhi = f^{-1} f_z = \sum_{k=1}^{\ell}\phi_k e_k$. 
 Then the following statements hold{\rm:}
 \begin{align}\label{basic integrability 1}
 &f_z = f \varPhi, \;\;\;\;\;f_{\bar z} = f{\bar \varPhi}, \\
 \label{basic integrability 2}
 &\sum_{k=1}^\ell \phi_k^2 =0, \\
 \label{basic integrability 2-b}
 &\sum_{k=1}^\ell |\phi_k|^2 =\frac{1}{2} e^{u}.
 \end{align}
 In particular, $\varPhi$ and $\bar{\varPhi}$ 
 satisfy the integrability condition
 \begin{equation}\label{basic integrability 3}
 \varPhi_{\bar z} - {\bar \varPhi}_z + 
 \lbrack {\bar \varPhi}, \varPhi \rbrack = 0.
 \end{equation}
 Conversely, let $\D$ be a simply-connected domain and 
 $\varPhi = \sum_{k=1}^\ell \phi_k e_k$ a non-zero
 $1$-form on $\D$ which takes values in the 
 complexification $\mathfrak{g}^{\mathbb{C}}$ of 
 $\mathfrak{g}$ satisfying the conditions 
 \eqref{basic integrability 2} and \eqref{basic integrability 3}. 
 Then for any initial condition in $G$ given at some base point in $\D$
 there exists a unique conformal immersion $f$ into $G$. 
 \end{Proposition}
 \begin{proof} 
 By \eqref{basic integrability 3} the integrability condition for the 
 equations \eqref{basic integrability 1} is satisfied. Therefore, there exists a
 map $f$ into the complexification 
 $G^\mathbb{C}$ of $G$ satisfying (\ref{basic integrability 1}). 
 Since the metric of $G$ is left-invariant, the conformality 
 and the non-degeneracy of a metric of $f$ 
 follows from \eqref{basic integrability 2} 
 and \eqref{basic integrability 2-b}.
 It is straightforward to verify that the
 partial derivatives of $f {\bar f}^{-1}$ vanish. Hence $ f {\bar f}^{-1}$
 is constant. If we have chosen an initial condition in $G$ for $f$, 
 then this constant matrix is $I$, the identity element in $G$.
 \end{proof}

\subsection{} 
Let $f:M\to{G}$ be a smooth map of a $2$-manifold $M$. 
Then $f$ induces a vector bundle $f^{*}TG$ over $M$ by
$$
f^{*}TG=\bigcup_{p\in M} T_{f(p)}G,
$$
where $TG$ is the tangent bundle of $G$.
The space of all smooth sections of $f^{*}TG$ is denoted by
$\varGamma(f^{*}TG)$. A section of $f^{*}TG$ is called
a \textit{vector field along} $f$.

The Levi-Civita 
connection $\nabla$ of $G$ induces a unique connection $\nabla^{f}$ 
on $f^{*}TG$ which satisfies the condition
$$
\nabla_{X}^{f}(V\circ{f})=
(\nabla_{df(X)}V)\circ{f},
$$
for all vector fields $X$ on $M$ 
and $V\in \varGamma(TG)$, see \cite[p.~4]{EL}.

Next assume that $M$ is a Riemannian $2$-manifold with 
a Riemannian metric $ds_M^2$. Then
the \textit{second fundamental form}
$\nabla{d}f$ of $f$ is defined by
\begin{equation}
(\nabla{d}f)(X,Y)
=\nabla^{f}_{X}{d}f(Y)-{d}f(\nabla^{M}_{X}Y),
\ \
X,Y \in \mathfrak{X}(M).
\end{equation}
Here $\nabla^M$ is the Levi-Civita connection of $(M,ds_M^2)$.
The \textit{tension field} $\tau(f)$ of $f$ is a section of 
$f^{*}TG$ defined by $\tau(f)=\mathrm{tr}(\nabla{d}f)$.

\subsection{}
For a smooth map $f:(M,ds^2_M)\to (G,ds^2)$, 
the \textit{energy} of $f$ is defined by 
$$
E(f)=\int_{M}\frac{1}{2}
|df|^{2}\>dA.
$$ 
A smooth map $f$ is a \textit{harmonic map}
provided that $f$ is a critical point of the energy 
under compactly supported variations. 
It is well known that $f$ is a harmonic map if and 
only if its tension field $\tau(f)$ is equal to zero, that is,
$$
\tau(f)=\mathrm{tr}(\nabla{d}f)=0.
$$ It should be remarked that the harmonicity of $f$ is invariant under 
conformal transformations of $M$. Thus the harmonicity makes sense for maps
from Riemann surfaces.

\subsection{}
 Let $f:M\to G$ be a conformal immersion of a Riemann surface $M$ into 
 $G$. Take a local complex coordinate $z = x+ i y$ and represent the 
 induced metric by $e^{u} dz d\bar z$. 
 It is a fundamental fact that the tension field 
 of $f$ is related to the \textit{mean curvature vector field} $\boldsymbol H$ 
 by:
 \begin{equation}\label{eq:meancurvature}
 \tau(f)=  2 \boldsymbol H. 
 \end{equation}
 This formula shows that a conformal immersion
$f:M\to{G}$ is a minimal surface 
if and only if it is harmonic.
 Since the metric is left-invariant the equation above can be 
 rephrased, using the vector fields $\varPhi= f^{-1} f_z$ and 
 $\bar \varPhi= f^{-1} f_{\bar z}$, as
\begin{equation}\label{eq:anotherstructure}
  \varPhi_{\bar z} + {\bar \varPhi}_z + \{\varPhi, \bar \varPhi\}
  = e^{u} f^{-1} \boldsymbol H,
\end{equation}
 where $\{\cdot, \cdot\}$ denotes the bilinear symmetric map defined by 
\begin{equation}\label{anticommutator}
 \{ X, Y\} = \nabla_X Y + \nabla_Y X
\end{equation}
for $X, Y \in \mathfrak g$.
 By \eqref{eq:anotherstructure},
 the harmonic map equation can be computed as
\footnote{
 Let $U:\mathfrak{g}\times \mathfrak{g}\to \mathfrak{g}$ denote the 
 symmetric bilinear map defined by
\begin{equation*}
 2\langle U(X,Y),Z\rangle
 =\langle X,[Z,Y]\rangle
 +
 \langle Y,[Z,X]\rangle,
 \ \ X,Y,Z 
 \in \mathfrak{g}.
\end{equation*}
 Then $2U(\varPhi, \bar \varPhi)$ and $\{\varPhi, \bar \varPhi\}$
 are the same, since 
 $$
 \langle \nabla_X Y + \nabla_Y X, Z\rangle 
 =\frac{1}{2} (X \langle Y, Z\rangle + Y \langle X, Z \rangle
 + X \langle Z, Y \rangle  
 -2 Z \langle X, Y\rangle -2 \langle [X, Z], Y \rangle 
 -2 \langle [Y, Z], X \rangle)$$
 and the left invariance of the vector fields implies that
 $X \langle Y, Z\rangle = Y \langle X, Z \rangle
 = X \langle Z, Y \rangle  =  Z \langle X, Y\rangle=0$.
 Moreover, since
 $$ 
 \langle X,[Z,Y]\rangle +
 \langle Y,[Z,X]\rangle = 
 -\langle X,\operatorname{ad}(Y)Z \rangle 
 -\langle Y,\operatorname{ad}(X)Z\rangle = 
 -\langle \operatorname{ad}^*(Y) X,Z \rangle 
 -\langle \operatorname{ad}^*(X)Y,Z\rangle,
 $$
 $2U(\varPhi, \bar \varPhi)$, $\{\varPhi, \bar \varPhi\}$ and 
 $-\operatorname{ad}^*(\varPhi) \bar \varPhi
 -\operatorname{ad}^*(\bar \varPhi) \varPhi$ 
 are the same (see \cite[Section 2.1]{BD} for another formulation 
 of harmonic maps into Lie groups with left-invariant metric).
 }
\begin{equation}\label{HME}
  \varPhi_{\bar z} +{\bar \varPhi}_z + \{\varPhi, \bar \varPhi\}=0.
\end{equation}
 Thus the Maurer-Cartan equation \eqref{basic integrability 3} 
 together with 
 the harmonic map equation (\ref{HME}) is equivalent to 
 \begin{equation}\label{harm+integ}
 2 \varPhi_{\bar z}+\{\varPhi,\bar {\varPhi}\}= [\varPhi,\bar {\varPhi}].
\end{equation}
 We summarize the above discussion as the following theorem.
\begin{Theorem}
 Let $f:M\to G$ be a conformal minimal immersion. Then 
 $\alpha=f^{-1}df=\varPhi{d}z+\bar {\varPhi}d\bar{z}$ satisfies 
 {\rm(\ref{conformalimmersion})} and 
 {\rm(\ref{harm+integ})}.
 Conversely, let $\mathbb{D}$ be a simply connected 
 domain and $\alpha=\varPhi{d}z+\bar {\varPhi}d\bar{z}$ a 
 $\mathfrak{g}$-valued 
 $1$-form on $\mathbb D$ 
 satisfying {\rm(\ref{conformalimmersion})} and 
 {\rm(\ref{harm+integ})}. 
 Then for any initial condition in $G$ 
 there exist a conformal minimal immersion $f:\mathbb{D}\to G$ such that $f^{-1}df=\alpha$. 
\end{Theorem}
\begin{proof}
 Let $\alpha = \varPhi{d}z+\bar {\varPhi}d\bar{z}$ be a
  $\mathfrak g$-valued $1$-form satisfying {\rm(\ref{conformalimmersion})} 
 and {\rm(\ref{harm+integ})}. 
 Then subtraction and addition of the complex conjugate 
 of \eqref{harm+integ} to itself
 gives the integrability condition \eqref{basic integrability 3} 
 and the harmonicity condition \eqref{HME}, respectively.
 Hence Proposition \ref{prop:Integrablity} implies 
 that there exists a conformal immersion $f$
 such that $f^{-1} d f =\alpha$. Since $f$ is harmonic, it is minimal.
\end{proof}
 In the study of harmonic maps of Riemann surfaces into 
 compact semi-simple Lie groups equipped with a
 bi-invariant Riemannian metric, the zero curvature representation is 
 the starting point of the loop group approach, see Segal \cite{Segal}, 
 Uhlenbeck \cite{Uhlenbeck}.  
 In case the metric on the target Lie group 
 is only left invariant we need to require the additional condition 
 $$
 \{\varPhi, \bar{\varPhi} \} =0,
 $$ 
 the so-called \textit{admissibility condition}. 

 It should be remarked that all the examples of minimal surfaces in 
 $\mathrm{Nil}_3$ studied in this paper do not satisfy the admissibility  
 condition. Thus we can not expect to generalize the Uhlenbeck-Segal 
 approach for harmonic maps into compact semi-simple Lie groups to maps 
 into more general Lie groups in a straightforward manner.
\section{The three-dimensional Heisenberg group $\mathrm{Nil}_3$}\label{sc:Preliminaries2}
\subsection{}
 We define a $1$-parameter family
 $\{\mathrm{Nil}_{3}(\tau)\}_{\tau \in \mathbb{R}}$ of $3$-dimensional 
 Lie groups 
 $$
\mathrm{Nil}_{3}(\tau)
=(\mathbb{R}^{3}(x_1,x_2,x_3),\cdot)
 $$
 with multiplication:
 $$
(x_1,x_2,x_3)\cdot
(\tilde{x}_1,\tilde{x}_2,\tilde{x}_3)
=
(x_{1}+\tilde{x}_1,
x_{2}+\tilde{x}_2,
x_{3}+\tilde{x}_3
+\tau
(x_1\tilde{x}_{2}-\tilde{x}_{1}x_{2})
).
 $$
 The unit element $\id$ 
 of $\mathrm{Nil}_3(\tau)$ is 
 $(0,0,0)$.
 The inverse element of $(x_1,x_2,x_3)$ is
 $-(x_1,x_2,x_3)$.
 Obviously, $\mathrm{Nil}_3(0)$ is the abelian group $(\mathbb{R}^3,+)$.
 The groups $\mathrm{Nil}_3(\tau)$ and $\mathrm{Nil}_3(\tau^{\prime})$
 are isomorphic if $\tau \tau^{\prime} \neq 0$.
 
\subsection{}
 The Lie algebra $\mathfrak{nil}_{3}(\tau)$ of $\mathrm{Nil}_3(\tau)$ is
 $\mathbb{R}^3$ with commutation relations:
\begin{equation}\label{1.1}
[e_1,e_2]=2\tau e_{3},
\ \
[e_2,e_3]=[e_3,e_1]=0
\end{equation}
 with respect to the natural basis 
 $e_{1}=(1,0,0)$, $e_{2}=(0,1,0)$, $e_{3}=(0,0,1)$.
 The formulas (\ref{1.1})
 imply that $\mathfrak{nil}_3(\tau)$ is nilpotent.
 The respective left translated vector fields of 
 $e_1$, $e_2$ and $e_3$ are
$$
E_{1}=\partial_{x_1}-\tau{x_2}
\partial_{x_3},
\
E_{2}=\partial_{x_2}+\tau{x_1}
\partial_{x_3}
\;\;
\mbox{and}
\;\;
E_{3}=\partial_{x_3}.
$$

 We define an inner product $\langle\cdot,\cdot\rangle$
 on $\mathfrak{nil}_3(\tau)$ so that
 $\{e_1,e_2,e_3\}$ is orthonormal with respect to it.
 Then the resulting 
 left-invariant Riemannian metric 
 $ds^2_\tau=\langle\cdot,\cdot\rangle_{\tau}$ 
 on $\mathrm{Nil}_3(\tau)$ is
\begin{equation}\label{1.2}
ds_{\tau}^2=(dx_1)^{2}+(dx_{2})^{2}+\omega_\tau
\otimes
\omega_\tau,
\end{equation}
 where
\begin{equation}\label{eq:omegatau}
\omega_\tau=dx_{3}+\tau(x_{2}dx_{1}-x_{1}dx_{2}).
\end{equation}
The $1$-form $\omega_\tau$ satisfies 
$d\omega_\tau \wedge \omega_\tau=-2\tau\> dx_{1}\wedge dx_{2}\wedge dx_{3}$.
Thus $\omega_\tau$ is a 
\textit{contact form} on $\mathrm{Nil}_3(\tau)$ if and only if $\tau\not=0$.

 The homogeneous Riemannian $3$-manifold 
 $(\mathrm{Nil}_3(\tau),ds_{\tau}^2)$ 
 is called the three-dimensional 
 \textit{Heisenberg group} if $\tau\not=0$.
 Note that $(\mathrm{Nil}_3(0),ds_0^2)$ 
 is the Euclidean $3$-space $\mathbb{E}^3$. 
 The homogeneous Riemannian $3$-manifold $(\mathrm{Nil}_3(1/2),ds_{1/2}^2)$ is 
 frequently referred to as the model space $\mathrm{Nil}_{3}$
 of the nilgeometry in the sense of Thurston, \cite{Thurston}.

\begin{Remark}
For $\tau\not=0$, $(\mathrm{Nil}_{3}(\tau),\omega_\tau)$ is a 
contact manifold, and the unit Killing vector field $E_3$ is 
the \textit{Reeb vector field} of this contact manifold. 
In particular $\mathrm{Nil}_3(1)$ is 
isometric to the \textit{Sasakian space form} 
$\mathbb{R}^{3}(-3)$ in the sense of contact Riemannian geometry,
\cite{Boyer, BG}.
\end{Remark}
 We orient $\mathrm{Nil}_{3}(\tau)$ 
 so that $\{E_1,E_2,E_3\}$ is a positive 
 orthonormal frame field. 
 Then the volume element $dv_\tau$ of 
 the oriented Riemannian $3$-manifold 
 $\mathrm{Nil}_{3}(\tau)$ 
 with respect to the metric $ds_\tau^2$ is 
 $dx_1\wedge dx_{2}\wedge dx_3$.
 The \textit{vector product operation} $\times$ with respect to 
 this orientation is defined by
$$
\langle X\times Y,Z\rangle_{\tau}=dv_\tau(X,Y,Z)
$$
 for all vector fields $X$, $Y$ and $Z$ on $\mathrm{Nil}_{3}(\tau)$. 

\subsection{}
The nilpotent Lie group $\mathrm{Nil}_3(\tau)$ is realized as a closed
subgroup of the general linear group $\mathrm{GL}_{4}\mathbb{R}$.
In fact, $\mathrm{Nil}_3(\tau)$ is imbedded in $\mathrm{GL}_{4}\mathbb{R}$ by
$\iota:\mathrm{Nil}_3(\tau)\to\mathrm{GL}_{4}\mathbb{R}$;
$$
 \iota(x_1,x_2,x_3)=
 e^{x_1}E_{11} + \sum_{i=2}^{4}E_{ii} + 2 \tau x_1 E_{23} 
 + (x_3 + \tau x_1 x_2)E_{24}+x_2 E_{34},
$$
 where $E_{ij}$ are $4$ by $4$ matrices with the $ij$-entry $1$,
 and all others $0$. Clearly $\iota$ is an injective Lie group 
 homomorphism. Thus  $\mathrm{Nil}_3(\tau)$ is identified with
 $\{\iota(x_1,x_2,x_3)\ \vert \ x_1,x_2,x_3 \in \mathbb{R}\}
 = \mathrm{Nil}_3(\tau)$.
 The Lie algebra $\mathfrak{nil}_3(\tau)$ 
 corresponds to
$$
\left\{
u_1 E_{11} + 2 \tau u_1 E_{23} + u_3 E_{24} +u_2E_{34}
\
\vert
\
u_1,u_2,u_3
\in \mathbb{R}
\right\}.
$$
 The orthonormal basis $\{e_1,e_2,e_3\}$ is identified with
$$
 e_{1}= E_{11} + 2 \tau E_{23}, \ e_{2}=  E_{34}\;\;\mbox{and}\;\;e_{3}=E_{24}.
$$
 The exponential map 
 $\exp:\mathfrak{nil}_{3}(\tau)\to \mathrm{Nil}_{3}(\tau)$ is given 
 explicitly by  
\begin{equation}\label{exp-map}
 \exp(x_{1}e_{1}+x_{2}e_{2}+x_{3}e_{3})=
 e^{x_1} E_{11} + \sum_{i=2}^{4}E_{ii} +2 \tau x_1 E_{23} 
 +(x_3 + \tau x_1 x_2) E_{24} + x_2 E_{34}.
\end{equation}
 This shows that $\exp$ is a diffeomorphism. Moreover the 
 inverse mapping $\exp^{-1}$ can be identified with the 
 global coordinate system $(x_1,x_2,x_3)$ of $\mathrm{Nil}_{3}(\tau)$.
 The coordinate system $(x_1,x_2,x_3)$ is called the 
 \textit{exponential coordinate system} of $\mathrm{Nil}_{3}(\tau)$.
 In this coordinate system the exponential map is the identity map.

\subsection{}
The Levi-Civita connection
$\nabla$ of $ds_\tau^2$ is
given by 
\begin{equation}\label{eq:Levi-Civita}
\begin{split}
\displaystyle 
\nabla_{e_1}{e_1}=0,\ \ 
\nabla_{e_1}{e_2}=\tau\>e_{3},\ \ 
\nabla_{e_1}{e_3}=-\tau\>e_{2},  \\
\nabla_{e_2}{e_1}=-\tau\>e_{3},\ \ 
\nabla_{e_2}{e_2}=0,\ \ 
\nabla_{e_2}{e_3}=\tau\>e_{1}, \\
\nabla_{e_3}{e_1}=-\tau\>e_{2},\ \ 
\nabla_{e_3}{e_2}=\tau\>e_{1},\ \ 
\nabla_{e_3}{e_3}=0.
\end{split}
\end{equation}
 The Riemannian curvature tensor 
 $R$ defined by $R(X,Y)=[\nabla_X,\nabla_Y]-\nabla_{[X,Y]}$ is given by
\begin{align*}
 R(X,Y)Z &= -3\tau^2
 \left(\langle Y,Z\rangle_{\tau}{X}-\langle Z,X\rangle_{\tau}{Y}\right)\\
 &\phantom{=}+ 4\tau^{2}\left(
 \omega_{\tau}(Y)\omega_{\tau}(Z)X-\omega_{\tau}(Z)\omega_{\tau}(X)Y\right)
 \\
 &\phantom{=}+4\tau^{2} 
 \left(\omega_{\tau}(X)\langle Y,Z\rangle_{\tau}-
 \omega_{\tau}(Y)\langle Z,X\rangle_{\tau}\right){e_3}.
 \end{align*}
 The Ricci tensor field $\mathrm{Ric}$ is given by
$$
\mathrm{Ric}=-\tau^2\langle\cdot,\cdot 
\rangle_{\tau}+4\tau^2\omega_{\tau} \otimes \omega_{\tau}.
$$
 The scalar curvature of $\mathrm{Nil}_{3}(\tau)$ is $\tau^{2}$.
 The symmetric bilinear map $\{\cdot, \cdot\}$ defined in \eqref{eq:anotherstructure}
 is explicitly given by 
\begin{equation}\label{2.11}
\{e_1,e_2\} =0,\ \
\{e_1,e_3\}=-2\tau e_{2}\ \ \mbox{and} \ \
\{e_2,e_3\}= 2\tau e_{1}.
\end{equation}
 Note that $\{\cdot,\cdot\}$ measures the 
 non right-invariance of the metric.
 In fact $\{\cdot, \cdot\}=0$ if and only if
 $ds_\tau^2$ is right invariant (and hence bi-invariant).
 The formulas (\ref{2.11}) imply that
 $ds_\tau^2$ is bi-invariant only when $\tau=0$.

\section{Surface theory in $\mathrm{Nil}_3$}\label{sc:SurfaceTheory}
\subsection{}
 Hereafter we study $\mathrm{Nil}_{3}(1/2)$ 
 for simplicity and denote the 
 space by $\mathrm{Nil}_3$. The metric $ds^2_{1/2}$ is simply 
 denoted by $ds^2=\langle\cdot,\cdot\rangle$ and 
 $\omega_{1/2}$ of \eqref{eq:omegatau} by $\omega$.

 Let $f:M\to \mathrm{Nil}_3$ be an immersion of a
 $2$-manifold. Our main interests are surfaces of 
 constant mean curvature 
 (and in particular minimal surfaces). 
 In the case where a surface has nonzero constant mean curvature, 
 we can assume without loss of generality that 
 $M$ is orientable and $f$ is a conformal
 immersion of a Riemann surface. 
 In the minimal surface case, if necessary, taking a double covering, 
 we may assume that $f$ is an orientable conformal immersion of 
 a Riemann surface. 
 Clearly, $\mathrm{Nil}_3$ is a three-dimensional Riemannian 
 spin manifold. Thus $f$ induces a spin structure on $M$. 
 Hereafter we will always use the induced spin structure on $M$.   

 As in Section \ref{sc:Preliminaries}, 
 we consider the $1$-form $\varPhi{d}z$ on a simply connected domain 
 $\mathbb{D}\subset M$ that takes values 
 in the complexification $\mathfrak{nil}_3^{\mathbb{C}}$ of the 
 Lie algebra $\mathfrak{nil}_3$. With respect to the natural 
basis $\{e_1,e_2,e_3\}$ of $\mathfrak{nil}_3$, we expand $\varPhi$ as 
$\varPhi =\sum_{k=1}^3 \phi_k e_k$
 and assume that \eqref{basic integrability 2} 
 and \eqref{basic integrability 2-b} are satisfied. 
 Then there exist
 complex valued functions
 $\psi_1$ and $\psi_2$ such that
 \footnote{In \cite[(8)]{BT:Sur-Lie}, the spinor
 representation  
 \begin{equation*}
 \phi_1 = \frac{i}{2}(\overline{\psi_2}^2 + \psi_1^2), \;\;
 \phi_2 = \frac{1}{2}(\overline{\psi_2}^2 - \psi_1^2), \;\;
 \phi_3 = \psi_1 \overline{\psi_2}
 \end{equation*}
 is used. The correspondence to ours representation is
 $$\psi_j \to \sqrt{2} \psi_j, \;  \;\;\mbox{and}\;\;
 e_1 \to e_2, \;\; e_2 \to e_1.
 $$ 
 Thus the sign of the unit normal also changes.}
 \begin{equation}\label{eq:phi}
 \phi_1 = (\overline{\psi_2})^2 - \psi_1^2, \;\;
 \phi_2 = i ((\overline{\psi_2})^2 + \psi_1^2), \;\;
 \phi_3 = 2 \psi_1 \overline{\psi_2},
 \end{equation}
 where $\overline{\psi_2}$ denotes the complex conjugate of $\psi_2$.
 It is easy to check that 
 $\psi_{1}\sdz$ and $\psi_{2}\sdzb$ 
 are well defined on $M$. 
 More precisely, $\psi_{1}\sdz$ and $\psi_{2}\sdzb$ are respective 
 sections  of the spin bundles $\varSigma$ and $\bar{\varSigma}$ over $M$, 
 see Appendix \ref{sc:spin}.
 The sections $\psi_{1}\sdz$ and 
 $\psi_{2}\sdzb$ are called the 
 \textit{generating spinors} of the conformally immersed 
 surface $f$ in $\mathrm{Nil}_3$. 
 The coefficient functions
 $\psi_1$ and $\psi_2$ are also called the generating spinors
 of $f$,  see \cite{BT:Sur-Lie}. Note that after a change of 
 coordinates the new generating spinors $\varphi_1, \varphi_2$ 
 are $\varphi_1(w) = \sqrt{z_{w}} \psi_1(z(w))$ and 
 $\varphi_2(w)=\sqrt{\bar z_w} \psi_2 (z(w))$.

 We would like to note that from this representation of
 $\varPhi$ it is straightforward to verify $f$ has 
 a branch point, that is $\varPhi$ and $\bar \varPhi$ 
 are linearly dependent, if and only if $\psi_1=\psi_2=0$ at the 
 point. Sometimes, we consider conformal immersions
 with branch points.
 Since we are interested in immersions, we will 
 only admit a nowhere dense set of branch points, if any.

 The conformal factor $e^{u}$ of the 
 induced metric $\langle df,df\rangle$ 
 can be expressed  by the spinors  $\psi_1, \psi_2$ via 
 formula \eqref{basic integrability 2-b}:
 \begin{equation}\label{eq:metric}
 e^{u} =4 (|\psi_1|^2+|\psi_2|^2)^2.
 \end{equation}

\begin{Remark}\label{rm:psizero}
 Let $f:M\to \mathrm{Nil}_3$ be a conformal immersion. 
 Then $\phi_3=2\psi_{1}\overline{\psi_2}$ can not vanish identically on 
 any open subset of $M$.
 In fact, if $\phi_3=0$, then $f$ is normal to $E_3$ everywhere. 
 Namely, $f$ is an integral surface of the contact distribution 
 defined by 
 the equation $\omega=0$. 
 However, since $\omega$ is a contact form on $\mathrm{Nil}_3$,
 this is impossible. (The maximum dimension of an integral manifold 
 is one.) In particular, for every conformal immersion $f: M \to 
 \Nil$, there exists an open and dense subset $M_f$ on which $\psi_1 \neq 
 0$ and $\psi_2 \neq 0$.    
\end{Remark}
\begin{Example}[Vertical plane]
Let $\Pi$ be an affine plane in $\mathrm{Nil}_3$ defined by 
$$
\Pi=\{(x_1,x_2,x_3) \ \vert \ ax_{1}+bx_{2}+c=0\}
$$ for some constants $a$, $b$ and $c$. 
Such a plane is called a \textit{vertical plane} in 
$\mathrm{Nil}_3$. One can see that every vertical plane is 
minimal in $\mathrm{Nil}_3$. Vertical planes are homogeneous and 
minimal Hopf cylinders. 
See Proposition \ref{homogenroussurfaces} and Theorem \ref{thm:ARholo}. 
Vertical planes are minimal and flat, but not totally geodesic. 
It should be emphasized that there are no totally umbilical surfaces
in $\mathrm{Nil}_3$, see \cite{Sanini, IKN}.
\end{Example}

\subsection{}
 Let $N$ denote the positively oriented 
 unit normal vector field along $f$.
 We then define an unnormalized normal vector field $L$ by
\begin{equation}\label{eq:unnormalizednormal}
 L = e^{u/2} N.
\end{equation}
 Note that $L\sdz \sdzb$ 
 is well defined on $M$. We call this section 
 the \textit{normal} of $f$.
 We also note that $e^{u/2} L = e^{u} N$ is given by 
 the vector product $f_x \times f_y$.

 Moreover, from \eqref{eq:phi}, the left translated vector field 
 $f^{-1}N$ 
 of the unit normal $N$ to $\mathfrak{nil}_3$
 can be represented by the spinors  $\psi_1$ and $\psi_2$:
 \begin{equation}\label{eq:Nell}
 f^{-1} N = \frac{1}{|\psi_1|^2+ |\psi_2|^2}
 \left( 
 2 \Re(\psi_1 \psi_2) e_1 +2 \Im(\psi_1 \psi_2)e_2 + 
 (|\psi_1|^2 -|\psi_2|^2)e_3
 \right),
 \end{equation}
 where $\Re$ and $\Im$ denote the real and 
 the imaginary parts of a complex number.
 Accordingly, the left translation of the unnormalized normal 
 $f^{-1} L = e^{u/2} f^{-1}N$ can be computed as 
%
 \begin{equation}\label{eq:L}
 f^{-1} L = 4 \Re(\psi_1 \psi_2) e_1 + 4 \Im(\psi_1 \psi_2)e_2 + 
 2 (|\psi_1|^2 -|\psi_2|^2)e_3.
 \end{equation}
 We define a function $h$ by
 $$
 h =\langle f^{-1} L, e_3 \rangle  =  2(|\psi_1|^2-|\psi_2|^2).
 $$ 
 Then we get a section $h\sdz\sdzb$ of 
 $\varSigma \otimes \bar{\varSigma}$. 
 This section is called the \textit{support} of $f$.  
 The coefficient function $h$ is called the \textit{support function} 
 of $f$ with respect to $z$. 

\begin{Remark}
 Let us denote by $\vartheta$ the angle between $N$ and the Reeb vector field $E_3$, then 
 $h$ is represented as $h=e^{u/2}\cos \vartheta$. The angle function $\vartheta$ is called 
 the \textit{contact angle} of $f$. One can check that 
 $h\sdz\sdzb=\cos \vartheta|df|$. Here $|df|=e^{u/2}\sdz\sdzb$ is a half density
 on $M$.
\end{Remark}
 From \eqref{eq:L}, we obtain the following
 Proposition.
\begin{Proposition}\label{heightzero}
 For a surface $f:\mathbb{D}\to \mathrm{Nil}_{3}$, the following 
 properties are equivalent{\rm:}
 \begin{enumerate}
 \item $f$ has the support function equal 
 zero at $p$ in $\mathbb D$, that is, 
 the support function $h$ of $f$ vanishes at $p$, $h(p) =0$.

 \item $E_3$ is tangent to $f$ at $p$.
 \end{enumerate}
\end{Proposition}
 Let $\pi: \Nil \to \R^2$ be the natural projection defined by 
 $\pi (x_1, x_2, x_3) = (x_1, x_2)$. We define a
 \textit{Hopf cylinder}
 by the inverse image of a plane curve under the projection $\pi$. 
 Hopf cylinders are flat and its mean curvature is 
 half of the curvature of the base curve.

 It is clear from the definition that 
 surfaces tangent to $E_3$ are Hopf cylinders, 
 \cite{BDI}. 
 Thus a surface which has zero support, 
 that is $h\equiv 0$, is a Hopf cylinder. 
 
 For later purpose we list some notion:
 A surface is called {\it vertical at $p$} in $M$ 
 if $E_3$ is tangent to $f$ at $p$ in $M$. A surface is 
 {\it vertical}, if it is vertical at all points $p$ in $M$. 
 A surface is called {\it nowhere vertical} if 
 it is nowhere tangent to $E_3$.
\subsection{} 
 Conformal immersions into $\Nil$ are
 characterized by the integrability condition
 \eqref{basic integrability 3} and the structure equation 
 \eqref{eq:anotherstructure}. Note, since the target space $\Nil$ 
 is three-dimensional, the mean curvature vector field $\boldsymbol H$
 in  \eqref{eq:anotherstructure} can be represented as 
 $$
 \boldsymbol H = H N,
 $$ 
 where $H$ is the mean curvature and $N$ is the unit normal.
 These equations are given by six equations
 for the functions $\phi_1$, $\phi_2$ and $\phi_3$ or, equivalently, 
 for the generating spinors $\psi_1$ and $\psi_2$,  see 
 \cite[$(18)$]{BT:Sur-Lie}. 
 Then the equations \eqref{basic integrability 3} and \eqref{eq:anotherstructure}
 are equivalent to the following \textit{nonlinear Dirac equation}, 
 that is,  
 \begin{equation}\label{Dirac1}
\slashed{D} \begin{pmatrix} 
\psi_1
\\ 
\psi_2
\end{pmatrix} 
:=
\begin{pmatrix}
\partial_{z}\psi_{2}+\mathcal{U}\psi_1
\\
-\partial_{\bar z}\psi_1+\mathcal{V}\psi_2
\end{pmatrix} 
=
\left(
\begin{array}{c}
0
\\
0
\end{array}
\right),
\end{equation}
 where 
 \begin{equation}\label{Dirac2}
 \mathcal U = \mathcal V = - \frac{H}{2}e^{u/2} + \frac{i}{4}h.  
 \end{equation}
 Here $H$, $e^{u}$ and $h$ are  
 the mean curvature, the conformal factor 
 and the support function for $f$
 respectively. More precisely by Remark \ref{rm:psizero},
 we have $\psi_1 \psi_2 \neq 0$ on an open dense subset $M_f$, 
 and on this subset, we show that \eqref{basic integrability 3} 
 and \eqref{eq:anotherstructure} together are equivalent with 
 the nonlinear Dirac equation. Thus we extend to $M$ by continuity.
 The complex function $\mathcal U (=\mathcal V)$ 
 is called the \textit{Dirac potential} of the nonlinear 
 Dirac operator $\slashed{D}$.

\begin{Remark}

\mbox{}
\begin{enumerate}
\item
 The above equivalence can be seen explicitly as follows:
 The coefficients of $e_1, e_2$ and $e_3$ in \eqref{basic integrability 3} 
 and \eqref{eq:anotherstructure} give six equations. 
 The equations given by the respective coefficients of $e_1$ and 
 $e_2$ in \eqref{basic integrability 3}
 and \eqref{eq:anotherstructure} together are equivalent to 
 the nonlinear Dirac equation. 
 Conversely, the equations given by the coefficients of $e_3$ 
 in \eqref{basic integrability 3} and \eqref{eq:anotherstructure}
 follow from  the nonlinear Dirac equation.
 Therefore the nonlinear Dirac equation is equivalent to 
 the integrability equation \eqref{basic integrability 3}
 together with the structure equation \eqref{eq:anotherstructure}. 

\item 
 To prove  the equations 
 \eqref{basic integrability 3} and \eqref{eq:anotherstructure} 
 from the nonlinear Dirac equation, we can choose 
 a real-valued function $H$ freely, 
 however, $e^{u/2}$ and $h$, which are the functions in the Dirac potential 
 $\mathcal U(=\mathcal V)$, \eqref{Dirac2}, 
 and solutions $\psi_j, (j=1, 2)$ of the nonlinear Dirac 
 equation need to 
 satisfy the special relation:
 $$
 e^{u/2} = 2(|\psi_1|^2+ |\psi_2|^2) \;\;\mbox{and}\;\; 
 h= 2(|\psi_1|^2 -|\psi_2|^2 ).
 $$
 Under this special condition, we derive the equations 
 \eqref{basic integrability 3} 
 and \eqref{eq:anotherstructure}. Moreover, up to an initial condition, 
 there exists an immersion into $\Nil$
 such that the conformal factor, the mean curvature and support function
 are $e^{u/2}$, $H$ and $h$, respectively.
 \end{enumerate}
\end{Remark}

\subsection{}
 The \textit{Hopf differential} $A \, dz^2$ is 
 the $(2,0)$-part of the second fundamental form for $f: 
 M \to \Nil$ 
 defined by
 $$
 A = \langle \nabla^f_{\partial_z} f_{z}, N \rangle.
 $$
 It is easy to see that $A$ can be expanded as 
 \begin{eqnarray*}
 A = \langle\nabla_{f_z} f_{z}, N \rangle 
   = \langle (f^{-1} f_z)_{z}, f^{-1} N \rangle  
   + \langle \textstyle \sum_{k, j} \phi_k\phi_j\nabla_{e_k} e_j, f^{-1} N \rangle,  
\end{eqnarray*}
 where $\phi_k, \phi_j$ are defined in \eqref{eq:phi}. 
 Then using the formulas in \eqref{eq:Levi-Civita} 
 and the $f^{-1} N$ in \eqref{eq:Nell}, the coefficient function  
 $A$ can be given explicitly as 
\begin{equation}\label{eq:Hopfdiff}
 A =  2 (\psi_1 (\overline{\psi_2})_z - \overline{\psi_2} (\psi_{1})_{z})
      +4 i \psi_1^2 (\overline{\psi_2})^2.
\end{equation}
Next, define $B$ as the complex valued function 
 \begin{equation}\label{eq:ARdiff}
 B = \frac{1}{4}(2 H + i) \tilde A, \;\; 
 \;\;\; \mbox{where} \;\;\; \tilde A = 
 A + \frac{\phi_3^2}{2 H +i}.
 \end{equation}
 Here $A$ and $\phi_3$ are the Hopf differential and 
 the $e_3$-component of $f^{-1} f_{z}$ for $f$ in $\Nil$, 
 respectively.

 The complex quadratic differential $\tilde A\, dz^2$ 
 will be called the \textit{Berdinsky-Taimanov differential} 
 \cite[Lemma 1]{Taimanov:Hei}. 
 Next we recall the \textit{Abresch-Rosenberg differential} of a surface 
 $f:M\to \mathrm{Nil}_{3}(\tau)$. It is the quadratic differential
 $Q dz^2$ given 
 by \cite{Fer-Mira2, Abresch-Rosenberg}:
 $$
 Qdz^2 = 2 (H+ i \tau)Adz^2+4\tau^2 \phi_3^2 \>dz^2. 
 $$
 It is clear that for $\tau =1/2$, the quadratic differential $4 Bdz^2$ 
 is the Abresch-Rosenberg differential.\footnote{
 In \cite[(20)]{BT:Sur-Lie}, 
 $$
 A = (\overline{\psi_2} \psi_{1 z} -\psi_1 \overline{\psi_2}_z)
 + i \psi_1^2 (\overline{\psi_2})^2.
 $$
 In many papers, $A_{AR}:=\tilde{A}dz^2$ is called 
 Abresch-Rosenberg differential. Taimanov uses the notation 
 $B=(2H+i)\tilde{A}/4$ and quotes Berdinsky's paper \cite{Ber:Heisenberg}. 
 Sometimes the differential $Bdz^2$ is called  Abresch-Rosenberg differential, 
 see e.g., \cite{Fer-Mira}.
 } 
\subsection{} 
 We are mainly interested in conformal immersions of constant mean curvature 
 into $\Nil$. Namely, our main 
 interest is the case, where both the Berdinsky-Taimanov differential 
 and the Abresch-Rosenberg differential are holomorphic.
 However, these differentials do not enter the nonlinear Dirac equations
 \eqref{Dirac1} and \eqref{Dirac2}. It is therefore fortunate that
 Berdinsky \cite{Ber:Heisenberg} found another system of partial 
 differential equations for the spinor field $\tilde \psi = (\psi_1, \psi_2)$ of 
 a surface $f$ which is actually equivalent with the nonlinear Dirac equations
 (for a proof see \cite{Taimanov:Hei})
 and where the quadratic differentials enter.
 We define a function $w$ using the Dirac potential $\mathcal U(=\mathcal V)$ 
 as 
 \begin{equation}\label{def-exp(w/2)}
 e^{w/2} =\mathcal{U} 
 = \mathcal V = - \frac{H}{2}e^{u/2} + \frac{i}{4}h.
 \end{equation}
 Here, to define the complex function $w$, 
 we need assume that the mean curvature $H$ and 
 the support function $h$ do not have any common zero.
 For nonzero constant mean curvature surfaces 
 this is no restriction, however, 
 for minimal surfaces, this assumption is equivalent to 
 that $h$ never vanish, that is, surfaces are
 nowhere vertical. 
 By Proposition \ref{heightzero}, minimal immersions with $h =0$ 
 everywhere are exactly the vertical planes.
 In what follows, we will thus always exclude vertical planes 
 from our discussions, and usually also work on open sets not 
 including any point where vertical.
 We can therefore use 
 the representation \eqref{def-exp(w/2)}.
 \begin{Theorem}[\cite{Ber:Heisenberg}]\label{thm:Berdinskysystem}
 Let $\D$ be a simply connected domain in $\C$ and $f: \D \rightarrow \Nil$
 a conformal immersion and $w$ is a complex function defined in 
 \eqref{def-exp(w/2)}. 
 Then the vector $\tilde \psi = (\psi_1, \psi_2)$
 satisfies the system of equations
 \begin{equation}\label{eq:Lax-Niltilde}
 \tilde \psi_z = \tilde \psi \tilde U , \;\;
 \tilde \psi_{\bar z} =  \tilde \psi \tilde V, \;\;
 \end{equation}
 where 
 \begin{equation}\label{eq:U-V1}
 \tilde U =
 \begin{pmatrix}
 \frac{1}{2} w_z  + \frac{1}{2} H_{z} e^{-w/2+u/2}&
 - e^{w/2} \\ 
 B e^{-w/2} & 0
 \end{pmatrix}, \;\;
 \tilde V =
 \begin{pmatrix}
 0 & - \bar B e^{-w/2}\\
 e^{w/2} & \frac{1}{2}w_{\bar z}+\frac{1}{2} H_{\bar z}e^{- w/2 + u/2}
 \end{pmatrix}.
 \end{equation}
 Conversely, every vector solution $\tilde \psi$ 
 to  \eqref{eq:Lax-Niltilde} with  \eqref{def-exp(w/2)} 
 and \eqref{eq:U-V1} is a solution to the nonlinear Dirac equation
 \eqref{Dirac1} with \eqref{Dirac2}. 
\end{Theorem}
\begin{proof}[Sketch of proof]
 Taking the derivative of the potential $\mathcal U = e^{w/2}$ 
 with respect to $z$, we have 
 $$
 \partial_z e^{w/2}
 =  -\frac{H_z}{2} e^{u/2} -\frac{H}{2} (e^{u/2})_z +\frac{i}{4}h_z.
 $$
 Using the explicit formulas for $e^{u/2}$ and $h$ 
 by $\psi_1$ and $\psi_2$, we have 
 $$
\frac{w_z}{2} e^{w/2} 
 = -\frac{H_z}{2}e^{u/2} 
 -\frac{2H+i}{2}\left(
 {\psi_2}_z \overline{\psi_{2}}+\psi_2 (\overline{\psi_{2}})_z\right) 
 -\frac{2H-i}{2}\left({\psi_1}_z \overline{\psi_{1}}+\psi_1 (\overline{\psi_{1}})_z\right). 
 $$ 
 From the nonlinear Dirac equation we can rephrase this as
 $$
\frac{w_z}{2} e^{w/2} 
 = -\frac{H_z}{2}e^{u/2} 
 -\frac{2H+i}{2} \psi_2 (\overline{\psi_2})_z
 -\frac{2H-i}{2}({\psi_1})_z \overline{\psi_{1}}
 -2 iH \psi_1 \overline{\psi_2}|\psi_2|^2.
 $$
 By multiplying the equation above by $\psi_1$ and 
 using the equation $e^{w/2} = -H e^{u/2}/2 + ih/4 =
 -(2H+i)|\psi_2|^2/2 -(2 H-i)|\psi_1|^2/2$, 
 we derive
 $$
 {\psi_1}_z =\left(\frac{w_z}{2}  + \frac{H_z}{2}e^{-w/2+u/2}\right) \psi_1
              + B e^{-w/2} \psi_2.
 $$
 Similarly, ${\psi_2}_z, {\psi_1}_{\bar z}$ and ${\psi_2}_{\bar z}$ 
 can be computed by the nonlinear Dirac equation.
 
 Conversely, let $\tilde \psi$ be a vector solution 
 of \eqref{eq:Lax-Niltilde}. Then the second column in $\tilde U$ and 
 the first column in $\tilde V$ produce the nonlinear Dirac equation.
\end{proof}
 The compatibility condition for \eqref{eq:Lax-Niltilde}, that is 
 $\tilde U_{\bar z} -\tilde V_{z} + [\tilde V, \tilde U] =0$, 
 gives the Gauss-Codazzi equations
 of a surface $f:\mathbb{D}\to \mathrm{Nil}_3$. 
 These are four equations, one for each matrix entry. 
 We obtain 
 \begin{equation}\label{GaussEquation}
 \frac{1}{2} w_{z \bar z}+e^{w}-|B|^2e^{-w}+\frac{1}{2}(H_{z \bar z}+p)e^{-w/2+u/2}=0,
\end{equation}
 where $p$ is $H_z (-w/2 +u/2)_{\bar z}$ for the $(1, 1)$-entry and
 $H_{\bar z} (-w/2 +u/2)_{z}$ for the $(2, 2)$-entry, respectively.
 Moreover, the remaining two equations are
 \begin{equation}\label{CodazziEquation}
\begin{split}
 \bar B_z  e^{-w/2}= -\frac{1}{2} \bar B H_z e^{-w+u/2} 
                    -\frac{1}{2} H_{\bar z}e^{u/2}, \\
B_{\bar z} e^{-w/2} = -\frac{1}{2} B H_{\bar z} e^{-w+u/2} 
                    -\frac{1}{2} H_{z}e^{u/2}.
\end{split}
 \end{equation}
 The Codazzi equations \eqref{CodazziEquation} 
 imply that $B$ is holomorphic if the surface is of constant mean curvature.
 However, we should emphasize that the holomorphicity of $B$ does 
 not imply  the constancy of the mean curvature. This situation is very
 different from the case of space forms.
 For a precise statement we refer to Appendix \ref{sc:Appendix}.

\begin{Remark}
 Let $w, B, H$ be solutions to the Gauss-Codazzi equations 
 \eqref{GaussEquation} and \eqref{CodazziEquation}.
 To obtain the immersion into $\Nil$, a vector solution $\tilde \psi
 = (\tilde \psi_1, \tilde \psi_2)$
 of \eqref{eq:Lax-Niltilde} is not enough; the complex function $e^{w/2}$
 also needs to satisfy
 $$
 e^{w/2} = - H(|\tilde \psi_1|^2+|\tilde \psi_2|^2) 
           + \frac{i}{2} (|\tilde \psi_1|^2-|\tilde \psi_2|^2).
 $$
 In Proposition \ref{prop:Ber-over} for constant mean curvature surfaces,  
 this will be rephrased in terms of equations for $w, H$ and $B$.
\end{Remark}
\section{Constant mean curvature surfaces in $\Nil$}\label{sc:CMCinNil}
\subsection{}
 Let $f$ be a conformal immersion in $\Nil$ as in the 
 preceding section and $\tilde  \psi = (\psi_1, \psi_2)$
 and $e^{w/2} = \mathcal U = \mathcal V \neq 0$ 
 the spinors generating $f$ and the Dirac potential, respectively.
 Then we have the equations $\tilde \psi_z = \tilde \psi \tilde U$ and 
 $\tilde \psi_{\bar z} =  \tilde \psi \tilde V$ as before.
 Take a fundamental system $\tilde{F}$ of solutions to this system, 
 we obtain the matrix differential equations
 \begin{equation}\label{eq:Lax-Ftilde}
 \tilde F_z = \tilde F \tilde U , \;\;
 \tilde F_{\bar z} =  \tilde F \tilde V. \;\;
 \end{equation}
 It will be convenient for us to replace this system of equations by some
 gauged system.
 Consider the $\mathrm{GL}_{2}\mathbb{C}$ valued function 
 $G= \di (e^{-w/4}, e^{- w/4})$
 and put $F := \tilde F G$, where $\di (a, b)$ denotes 
 the diagonal $\GL$ matrix with entries $a, b$.
 Then the complex matrix $F$ satisfies the 
 equations
 \begin{equation}\label{eq:Lax-Nil}
 F_z = F U, \;\;F_{\bar z} = F V, 
 \end{equation}
 where $U = G^{-1} \tilde U G + G^{-1} G_z$ and 
 $V = G^{-1} \tilde V G + G^{-1} G_{\bar z}$. 
\subsection{}
 We define a family of Maurer-Cartan forms 
 $\alpha^{\lambda}$, parametrized by $U$ and $V$ 
 and the spectral parameter $\lambda \in \C^\times(:=\mathbb{C}\setminus\{0\})$
 as follows:
 \begin{equation}\label{eq:alpha}
 \alpha^{\lambda} := U^{\lambda} dz + V^{\lambda}d\bar z,
 \end{equation}
 where
 \begin{equation}\label{eq:U-V}
 U^{\lambda} =
 \begin{pmatrix}
 \frac{1}{4} w_z +\frac{1}{2} H_z e^{-w/2+u/2} 
 &-\lambda^{-1}e^{w/2} \\[0.1cm]
  \lambda^{-1} B e^{-w/2} & - \frac{1}{4} w_z
 \end{pmatrix}, \;\;
 V^{\lambda} =
 \begin{pmatrix}
 - \frac{1}{4} w_{\bar z} 
 &  -\lambda \bar B e^{-w/2}\\[0.1cm]
 \lambda e^{w/2} & 
 \frac{1}{4} w_{\bar z} +\frac{1}{2} H_{\bar z} e^{-w/2+u/2} 
 \end{pmatrix}.
 \end{equation}
 We note that $U^{\lambda}|_{\lambda =1} = U$ and $V^{\lambda}|_{\lambda =1} = V$.
 Similar to what happens in space forms, a surface $f$ in $\Nil$ of constant 
 mean curvature can be characterized as follows:
 \begin{Theorem}\label{thm:CMCcharact}
 Let $f : \D \to \Nil$ be a conformal immersion and $\alpha^{\lambda}$
 the $1$-form defined in \eqref{eq:alpha}.
 Then the following statements are mutually equivalent{\rm:}
 \begin{enumerate}
 \item $f$ has constant mean curvature.
 \item $d + \alpha^{\lambda}$ is a family of flat connections on $\D \times 
 \GL$.
 \item The map $\ad (F) \sigma_3$ from $\D$ to the semi-Riemannian symmetric 
 space $\GL /\di$ is harmonic.
 \end{enumerate}
 Here $\sigma_3$ denotes the diagonal matrix with entries $1, -1$ 
 and $\di$ denotes the diagonal subgroup $\mathrm{GL}_{1}\mathbb{C}
 \times \mathrm{GL}_{1}\mathbb{C}$ of $\mathrm{GL}_{2}\mathbb{C}$.
 \end{Theorem}
 \begin{proof} We start by writing out  the conditions describing that 
 $d + \alpha^{\lambda}$ is a family of flat connections on $\D \times 
 \GL$.
 It is straightforward to see that $d + \alpha^{\lambda}$ is flat 
 for all $\lambda\in \C^\times$ if and only if the equation
\begin{equation}\label{integrability}
 (U^{\lambda})_{\bar z}- (V^{\lambda})_z+[V^{\lambda}, U^{\lambda}]
 =0.
 \end{equation}
 is satisfied for all $\lambda\in \C^\times$.
 The coefficients of $\lambda^{-1}, \lambda^{0}$ and $\lambda$ 
 of \eqref{integrability} can be computed explicitly as follows:
 \begin{eqnarray}
 &\mbox{$\lambda^{-1}$-part:}\;\; \frac{1}{2} H_{\bar z} e^{u/2} =0, 
 \;\; B_{\bar z}+\frac{1}{2} B H_{\bar z}e^{-w/2+u/2}=0, \label{eq:codazzi1}\\
 & \mbox{$\lambda^{0}$-part:}\;\;
 \frac{1}{2} w_{z \bar z} + e^{w} -|B|^2 e^{-w}+\frac{1}{2}(H_{z \bar z}
 +p ) e^{-w/2 + u/2}  =0, \label{eq:structure}\\
 &\mbox{$\lambda$-part:} \;\;  
 \bar B_z+\frac{1}{2} \bar B H_z e^{-w/2+u/2}=0, 
  \;\;\frac{1}{2} H_z e^{u/2} =0,\label{eq:codazzi2}
 \end{eqnarray}
 where $p$ is $H_z (-w/2 +u/2)_{\bar z}$ for the $(1, 1)$-entry and
 $H_{\bar z} (-w/2 +u/2)_{z}$ for the $(2, 2)$-entry, respectively.
 Since the equations in \eqref{eq:structure} are structure equations for 
 the immersion $f$, these are always satisfied, which in fact is equivalent
 to \eqref{GaussEquation}. 

 $(1)\Rightarrow (2)$:
 Assume now that $f$ has constant mean curvature. 
 Then, as already mentioned earlier, 
 $\tilde A = (A + \phi_3^2/(2 H+i))$ is holomorphic, 
 \cite[Corollary2, Proposition3]{BT:Sur-Lie}. 
 Thus $B =(2 H +i) \tilde A/4$ is holomorphic as well.
 Clearly now, all equations characterizing the flatness of 
 $d + \alpha^{\lambda}$ are satisfied.

 $(2)\Rightarrow(1)$: 
 Assume now that $d + \alpha^{\lambda}$ is flat. Then it is easy to see that
 this implies $H$ is constant.
 
 $(1)\Leftrightarrow (3)$: Assume now that the first two statements of the theorem are satisfied.
 Then the coefficient matrices $U^{\lambda}$ and $V^{\lambda}$ actually 
 have trace $=0$ and $\alpha^{\lambda}$ describes the Maurer-Cartan form 
 of a harmonic map of the symmetric space $\SL / \di$ as in 
 \cite[Proposition3.3]{DPW}.
 Conversely, if $\ad (F) \sigma_3 $ is a harmonic map into
 $\GL / \di = \SL / \di$, then  \cite{DPW} shows that  $d + \alpha^{\lambda}$ is flat. 
 Note that the diagonal subgroup $\di$ of $\SL$ is 
 $\{\di(\gamma,\gamma^{-1})\ \vert \ \gamma\in \mathbb{C}^\times\}$ 
 that is isomorphic to $\mathrm{GL}_{1}\mathbb{C}$.
\end{proof}
\begin{Remark}\label{SpaceGeo}
 The semi-Riemannian symmetric space $\mathrm{SL}_{2}\mathbb{C}/
 \mathrm{GL}_{1}\mathbb{C}$ is identified with 
 the space of all oriented geodesics in the 
 hyperbolic $3$-space $\mathbb{H}^3$. The pairwise hyperbolic Gauss 
 maps of constant mean curvature surfaces in $\mathbb{H}^3$ are 
 Lagrangian harmonic maps into the indefinite K{\"a}hler symmetric space 
 $\mathrm{SL}_{2}\mathbb{C}/\mathrm{GL}_{1}\mathbb{C}$, \cite{DIK}.
\end{Remark}
 From the list of equations characterizing 
 the flatness of  $d + \alpha^{\lambda}$, we obtain the following.
\begin{Corollary} 
 Let $f : \D \to \Nil$ be a conformal immersion.
 If $f$ has constant mean curvature, then $B$ is holomorphic and
\begin{equation}\label{integrCMC}
 w_{z \bar z} + 2e^{w} -2|B|^2 e^{-w}=0
\end{equation}
 holds. The equation {\rm(\ref{integrCMC})} is the Gauss equation of the 
 constant mean curvature surface.
\end{Corollary}

\begin{Remark}
\mbox{}
\begin{enumerate}
\item By what was said earlier, the converse is almost true. 
Actually there is only one counter example.
In particular, the converse to the statement in the corollary is true, if 
the spinors $\psi_1 $ and $\psi_2$ satisfy $|\psi_1| \neq |\psi_2|$.

\item 
 If, in the setting of the corollary above, we assume $B$ to be holomorphic,
 then the solution to the elliptic equation \eqref{integrCMC} 
 produces an analytic solution $w$.
 Inserting this into \eqref{eq:Lax-Ftilde} we see that 
 the spinors $\psi_j$, $ j = 1,2$
 are analytic. Therefore, the condition in (1) above follows, 
 if $\psi_1$ and $\psi_2$ have a different absolute value at 
 least at one point of the domain $\D$.
\end{enumerate}
\end{Remark}

\subsection{}\label{subsc:associated}
 If $f$ is a constant mean curvature immersion into $\Nil$, 
 then the corresponding maps $\ad (F) \sigma_3$, where $F$ satisfies 
 $F^{-1} d F = \alpha^{\lambda}, \lambda \in \mathbb S^1$,
 into $\SL / \di$ all are harmonic. 
 They form the \textit{associated family} of harmonic maps.
 Tracing back all steps carried out so far, starting from $f$ and ending up
 at $F$, one can define maps $f^{\lambda} : \D \rightarrow \Nil^{\C}$ 
 into the complexification $\Nil^\mathbb{C}$ of $\Nil$. 
 We call the $1$-parameter family of maps $\{f^\lambda\}_{\lambda 
 \in \mathbb S^1}$ the 
 \textit{associated family} of $f$. 
 From the definition it is clear that the associated family $f^{\lambda}$
 has the invariant support function $e^{w/2}$ 
 and the varying Abresch-Rosenberg differential $4 B^{\lambda} dz^2$, 
 that is, $B^{\lambda} = 
 \lambda^{-2} B$.
 One could hope that all these maps in the associated family 
 actually have values in $\Nil$ 
 and are of constant mean curvature. This turns out not to be the case.

 The reason is that for constant mean curvature surfaces into $\Nil$ 
 there exists a special formula, due to Berdinsky 
 \cite{Ber:Note, Ber:Heisenberg}.
\begin{Proposition}[\cite{Ber:Note, Ber:Heisenberg}]\label{prop:Ber-over}
 Let $f:\D \rightarrow \Nil$ be a conformal immersion of constant mean curvature.
 Then the imaginary part of $w$ is constant or one of the following 
 three formulas holds:
\begin{equation*}
e^{( \bar{w} - w)/2} = \frac{2H+i}{2H-i}, \;\;\;\;e^{( \bar{w} - w)/2}  = \frac{2H-i}{2H+i}
\end{equation*}
or 
\begin{equation}\label{eq:Berd3}
\left|r + \bar{r} \frac{B}{|e^{w}|}\right|^2 
 = -st \left( 1 - \frac{|B|^2}{|e^{2w}|}\right)^2,
\end{equation}
 where $B$ and $w$ are as above and 
 $r = -\frac{1}{2}(2H + i){(\bar{w} - w)}_z, 
 s = (2H+i)e^{{\bar{w}/2}} -  (2H-i)e^{w/2}$
 and $t = (2H+i) e^{w/2} -(2H-i)e^{{\bar{w}/2}}$.
\end{Proposition}
\begin{proof}[Sketch of proof]
 Let us consider the equation 
 $$
 e^{w/2} = -H(|\psi_1|^2 + |\psi_2|^2) + \frac{i}{2}(|\psi_1|^2- |\psi_2|^2).
 $$ 
 This equation is rewritten equivalently in the form 
 \begin{equation}\label{eq:hath}
 I = \bar{\psi}^t \hat{h} \psi \;\;\mbox{with}\;\; 
 \hat h =
 \begin{pmatrix}
 \frac{1}{2}(-2H +i) e^{-w/2} & 0 \\ 0 & \frac{1}{2}(-2 H - i)e^{-w/2} 
 \end{pmatrix}.
 \end{equation}
 This equation is differentiated for $z$ and for $ \bar{z}$.
 The resulting equations are
 \begin{equation}\label{eq:Br1}
 \mathrm{tr}\>
 \left(
 \frac{1}{2}e^{-w/2}
 \begin{pmatrix}0 & \frac{t B}{|e^w|} \\ t & r
 \end{pmatrix} 
 \begin{pmatrix}
 |\psi_1|^2 & \psi_1 \overline{\psi_2} \\
 \overline{\psi_1} \psi_2 & |\psi_2|^2 
 \end{pmatrix}
 \right) =0, 
 \end{equation}
 \begin{equation}\label{eq:Br2}
 \mathrm{tr}\>
 \left(
 \frac{1}{2}e^{-w/2}
 \begin{pmatrix}- \bar r & s \\
 \frac{s \bar B}{|e^w|} & 0
 \end{pmatrix} 
 \begin{pmatrix}
 |\psi_1|^2 & \psi_1 \overline{\psi_2} \\
 \overline{\psi_1} \psi_2 & |\psi_2|^2 
 \end{pmatrix}
 \right) =0,
 \end{equation}
 where $r =-\tfrac{1}{2}(2H +i)(\bar w- w)_z$, 
 $t =(2H+i)e^{w/2} -(2H-i)e^{\bar w/2}$
 and $s =(2H+i)e^{\bar w/2} -(2H-i)e^{w/2}$.
 Conversely, from these latter two equations one obtains
  $ce^{w/2} = -H(|\psi_1|^2 + |\psi_2|^2) + i(|\psi_1|^2- |\psi_2|^2)/2$, 
 where $c$ is a complex constant. 
 One can normalize things or manipulate things so 
 that this constant can be removed. 
 Thus \eqref{eq:hath} is equivalent with 
 \eqref{eq:Br1} and \eqref{eq:Br2}.
 The equations  \eqref{eq:Br1} and \eqref{eq:Br2}
 are then reformulated equivalently in the form:
$$
\frac{t B}{|e ^w|} \xi + t \bar{\xi} + r |\xi|^2 = 0
\;\;\mbox{and}\;\;
\frac{s B}{|e ^w|} \xi + s \bar{\xi}+r =0,
$$
 where $\xi = \psi_2/\psi_1$, 
 which can be done without loss of generality since otherwise $\psi_1 = 0$ 
 identically and the Berdinsky system would also yield $\psi_2 =0$.
 The next conclusion in \cite{Ber:Heisenberg} requires to assume that 
 $s$  does not vanish identically.
 Therefore we need to admit the case $s = 0$ and 
 continue with the assumption $s \neq 0$. 
 Inserting the second equation into the first yields
 $r (|\xi|^2 - t/s) = 0$. We thus need to admit the case $r=0$.
 It is easy to see that this latter condition implies that the imaginary 
 part of $w$ is constant.
 This same statement follows from  $s = 0$.
 Now we assume $s \neq 0$ and $r \neq 0$ 
 and obtain (as claimed in \cite{Ber:Heisenberg})
 $|\xi|^2 = t/s$. 
 Inserting this into the first of the last two equations above we see that 
 these two equations are complex conjugates of each other now if $t \neq 0$. 
 But $t = 0$ implies again that the imaginary part of $w$ is constant.
 Finally, solving for $\xi$ in the second equation above 
 and inserting into the first one now obtains
 the equation 
 $$ 
 \left(1 - \frac{|B|^2}{|e^{w}|^2}\right) \bar{\xi} = 
 -\frac{1}{s}\left(r + \bar{r} \frac{B}{|e^{w}|}\right).
 $$
 Taking absolute values here yields the last of our equations. 
\end{proof}
 As a corollary to this result we obtain the following.
\begin{Corollary}\label{coro:CMCnotinNil}
 Let $f:\D \rightarrow \Nil$ be a conformal 
 immersion of constant mean curvature.
 If the associated family 
 $f^{\lambda}$ of immersions into $\Nil^{\C}$ 
 as defined above actually has values in $\Nil$ 
 for all $\lambda \in \mathbb{S}^1$, 
 then $\bar{w} - w$ is constant and the surfaces are minimal.
\end{Corollary}
\begin{proof}
 Assume the last of the equations is satisfied.
 Introducing $\lambda$ as in this paper can also be interpreted, 
 like for constant mean curvature surfaces in $\R^3$ by Bonnet, 
 as replacing $B$ by $ \lambda^{-2}B$, $\lambda \in \mathbb{S}^1$.
 This is an immediate consequence of \eqref{integrCMC}.
 Let $\psi_j$ be the $\lambda$-dependent solutions 
 to the equations \eqref{eq:Lax-Nil}.
 Then $f^{\lambda}$ is defined from these $\psi_j$ as $f$ was defined in the case 
 $\lambda = 1$. Thus the Berdinsky system associated with $f^{\lambda}$ is 
 a constant mean curvature system and the immersions $f^{\lambda}$ 
 all are constant mean curvature immersions into $\Nil$.
 Therefore, all the quantities associated with these immersions  satisfy 
 the Berdinsky equation \eqref{eq:Berd3}. As a consequence of our assumptions,  
 this equation needs to be satisfied for all $B^{\lambda}$.
 But replacing $B$ by $ \lambda^{-2}B = B^{\lambda}$ in \eqref{eq:Berd3} 
 we see that the required equality 
 only holds for all $\lambda \in \mathbb{S}^1$ if and only if the 
 function $r$ occurring in \eqref{eq:Berd3}
 vanishes identically.  Moreover, the vanishing of $r$ implies
 that $ \bar{w} -w$ is antiholomorphic, 
 and since this function only attains values in $i\R$, it is constant.
 Let us consider next the Gauss equation \eqref{integrCMC}.
 Since the imaginary part 
 $\Im{w} = \theta_0$ of $w$ 
 is constant, the term $w_{z \bar{z}}$ is real. 
 Therefore the imaginary part of 
 $e^w - B \bar{B} e^{-w}$ vanishes. A simple computation shows that this implies 
 $( e^{2 \Re{w}} + B \bar{B}) \sin(\theta_0) = 0 $, whence $\sin(\theta_0) = 0$ and 
 $\theta_0$ is an integral multiple of $\pi$. 
 As a consequence, $e^{w/2} = e^{\Re{w}/2} e^{ik\pi/2}$.
 If $k$ is odd, then $e^{w/2}$ is purely imaginary and $H=0$ follows. 
 If $k$ is even, then $e^{w/2}$ is real. 
 This implies $|\psi_1|^2 = |\psi_2|^2$, which shows that it is a surface of 
 non constant mean curvature, see Appendix \ref{sc:Appendix}.
\end{proof}
\begin{Remark}
 We have just shown that associated families 
 of ``real'' constant mean curvature surfaces in $\Nil$
 can only be minimal. We will show in the next section that actually every
 minimal surface is a member of an associated family of minimal surfaces 
 in $\Nil$.
\end{Remark}
\section{Characterizations of minimal surfaces in $\Nil$}\label{sc:minimal}
\subsection{}
 We recall the beginning of section \ref{sc:CMCinNil}.
 In particular, we consider the family of Maurer-Cartan forms $\alpha^{\lambda}$
 \begin{equation}\label{eq:alpha2}
 \alpha^{\lambda} := U^{\lambda} dz + V^{\lambda}d\bar z, \ \ 
 \lambda \in \mathbb{S}^1,
 \end{equation}
 where $U^{\lambda}$ and $V^{\lambda}$ are defined in \eqref{eq:U-V}.
 For surfaces of constant mean curvature these expressions have a
 particularly simple form:
 \begin{equation}\label{eq:U-VCMC}
 U(\lambda) (:=U^{\lambda}) =
 \begin{pmatrix}
 \frac{1}{4} w_z  
 &-\lambda^{-1}e^{w/2} \\[0.1cm]
  \lambda^{-1} B e^{-w/2} & - \frac{1}{4} w_z
 \end{pmatrix}, \;\;
 V(\lambda) (:= V^{\lambda}) =
 \begin{pmatrix}
 - \frac{1}{4} w_{\bar z} 
 &  -\lambda \bar B e^{-w/2}\\[0.1cm]
 \lambda e^{w/2} & 
 \frac{1}{4} w_{\bar z} 
 \end{pmatrix}.
 \end{equation}
 Minimal surfaces can be easily characterized among all 
 constant mean curvature surfaces in the following manner.
 \begin{Lemma}\label{thm:min-sym}
 Let $f$ be a surface of constant mean curvature in $\Nil$. 
 Then the following statements are mutually equivalent{\rm:}
 \begin{enumerate}
 \item $f$ is a minimal surface.
 \item  $e^{w/2} = - \frac{H}{2} e^{u/2} + \frac{i}{4}h$ is purely imaginary.
 \item The matrices $U(\lambda)$ and $V(\lambda)$ satisfy
  \begin{equation}\label{eq:U-Vsymm}
 V(\lambda) = - \sigma_3 \overline{ U(1/\bar \lambda)}^t \sigma_3,
 \;\;\mbox{where}\;\;
  \sigma_3 = \di (1, -1).
 \end{equation}
 \end{enumerate}
 In particular, for a constant mean curvature surface $f$, 
 the Maurer-Cartan form $ \alpha^{\lambda}$
 takes values in the real Lie subalgebra $\isu$
 of $\mathfrak{sl}_{2}\mathbb{C}$ if and only if $f$ is minimal{\rm;}
$$
\isu =
\left\{
\left.
\left(
\begin{array}{cc}
ai & b\\
\bar{b} & -ai
\end{array}
\right)
\
\right|
\ a \in \mathbb{R},\> b \in \mathbb{C}
\>
\right\}.
$$ 
\end{Lemma}
\subsection{}
 As is well known, constancy of the mean curvature of 
 surfaces in three-dimensional space forms is equivalent to 
 the holomorphicity of the Hopf differential. 
 Moreover constancy of the mean curvature is 
 characterized by harmonicity of appropriate Gauss maps.
 
 To obtain another characterization of minimal surfaces 
 we will introduce the notion of Gauss map for surfaces in 
 $\Nil$.  Let $N$ be the unit normal vector field along the surface $f$ 
 and $f^{-1} N$ the left translation of $N$.
 
 We identify the Lie algebra $\mathfrak{nil}_3$ of $\Nil$ 
 with Euclidean $3$-space 
 $\mathbb E^3$ via the natural basis $\{e_1, e_2,e_3\}$.
 Under this identification, 
 the map $f^{-1} N$ can be considered as a map into the 
 unit two-sphere $\mathbb S^2\subset \mathfrak{nil}_3$. 
 We now consider the \textit{normal Gauss map} $g$ of the surface $f$
 in $\Nil$, \cite{Ino:MiniHeisenberg, Daniel:GaussHeisenberg}:
 The map $g$ is defined as the composition of the stereographic 
 projection $\pi$ from the south pole with
 $f^{-1} N$,  that is, $g = \pi \circ f^{-1} N: \D \to \C \cup \{\infty\}$
 and thus, applying the stereographic projection to  $f^{-1} N$ 
 defined in \eqref{eq:Nell}, we obtain
 \begin{equation}\label{eq:Normal}
 g= \frac{\psi_2}{\overline{\psi_1}} \;.
 \end{equation}
 Note that the unit normal $N$ is represented 
 in terms of the normal Gauss map $g$ as 
 \begin{equation}\label{normalGausstounitnormal}
 f^{-1}N=\frac{1}{1+|g|^2}
 \left(
 2\Re (g) e_1+2\Im (g) e_2+(1-|g|^2)e_3
 \right).
 \end{equation}
 The formula (\ref{normalGausstounitnormal}) implies that 
 $f$ is nowhere vertical if and only if $|g|<1$ or $|g|>1$.
\begin{Remark} 
 \mbox{}
\begin{enumerate}
 \item  If $|g|>1$, then the $e_3$-component of $f^{-1} N$ 
 has a negative sign.
 Therefore such surfaces are called ``downward''. Analogously
 $|g|<1$ the surfaces are called ``upward''.

 \item  The normal Gauss map of a vertical plane satisfies 
 $|g|\equiv 1$. Conversely if the normal Gauss map $g$ of a conformal 
 minimal immersion satisfies $|g| \equiv 1$, 
 then it is a vertical plane.  
\end{enumerate}
\end{Remark} 
\subsection{}
 We have seen in Theorem \ref{thm:CMCcharact}
 and Remark \ref{SpaceGeo}, there exist harmonic maps 
 into the semi-Riemannian symmetric space 
 $\mathrm{SL}_{2}\mathbb{C}/\mathrm{GL}_{1}\mathbb{C}$ 
 associated to constant mean curvature surface in 
 $\Nil$. 
 In view of Lemma \ref{thm:min-sym}, one would expect 
 that minimal surfaces can be characterized by 
 harmonic maps into semi-Riemannian symmetric spaces associated 
 to the real Lie subgroup
$$
 \mathrm{SU}_{1,1}=
 \left\{
 \begin{pmatrix}
 a &  \bar{b}\\
 b & \bar{a}
 \end{pmatrix}
 \in \mathrm{SL}_{2}\mathbb{C}
 \right\}
$$ 
 of $\mathrm{SL}_{2}\mathbb{C}$ 
 with Lie algebra $\isu$. 
 For this purpose we recall a Riemannian 
 symmetric space representation 
 of the hyperbolic $2$-space $\mathbb{H}^2$. Note that 
 the symmetric space $\mathrm{SL}_{2}\mathbb{C}/\mathrm{GL}_{1}\mathbb{C}$
 is regarded as a ``complexification'' of $\mathbb{H}^2$. 
 Since $\mathrm{SL}_{2}\mathbb{C}/\mathrm{GL}_{1}\mathbb{C}
 =\mathrm{SU}_{1,1}^{\mathbb{C}}/\mathrm{U}_{1}^{\mathbb{C}}$.

 Let us equip the Lie algebra $\isu$ with the 
 following Lorentz scalar product $\langle \cdot,\cdot\rangle_m$ 
$$
 \langle X,Y \rangle_{m} =2\mathrm{tr}\>(XY), \ \ X,Y \in \isu.
$$
 Then $\isu$ is identified with Minkowski $3$-space 
 $\mathbb{E}^{2,1}$ as an indefinite scalar product space.
 The hyperbolic $2$-space $\mathbb{H}^2$ of constant curvature 
 $-1$ is realized in $\isu$ as 
 a quadric
$$
\mathbb{H}^2=\{X \in \isu
\
\vert
\ \langle X,X \rangle_m=-1,\ \
\langle X,i\sigma_3\rangle_{m}<0
\}.
$$
 The Lie group $\mathrm{SU}_{1,1}$ acts transitively and 
 isometrically on $\mathbb{H}^2$ via the $\ad$-action.
 The isotropy subgroup at $i\sigma_{3}/2$ is 
 $\mathrm{U}_1$ that is the Lie subgroup of $\mathrm{SU}_{1,1}$ 
 consisting of diagonal matrices. The resulting homogeneous 
 Riemannian $2$-space 
 $\mathbb{H}^{2}=\mathrm{SU}_{1,1}/\mathrm{U}_1$
 is a Riemannian symmetric space with involution 
 $\sigma=\ad (\sigma_3)$. 

 Next we recall the stereographic projection from the 
 hyperbolic $2$-space $\mathbb{H}^2\subset \mathbb{E}^{2,1}$ onto the 
 Poincar{\'e} disc $\mathcal{D}\subset \mathbb{C}$.
 We identify $\isu$ with Minkowski $3$-space 
 $\mathbb{E}^{2,1}$ by the correspondence:
$$
\frac{1}{2}
\left(
\begin{array}{cc}
ri & -p -qi\\
-p+qi & -ri
\end{array} 
\right) \in \isu
\longleftrightarrow (p,q,r) \in \mathbb{E}^{2,1}.
$$
 Under this identification,
 the stereographic projection 
 $\pi_{h}:\mathbb{H}^2\to \mathcal{D}$ with base point 
 $-i\sigma_{3}/2$ is given explicitly by
\begin{equation}\label{eq:Minpro}
\pi_{h}(p,q,r)=\frac{1}{1+r}(p+qi).
\end{equation}
 The inverse mapping of $\pi_{h}^{-1}$ is computed as
$$
\pi_{h}^{-1}(z)=\frac{1}{1-|z|^2}
\left(
2\Re(z),2\Im(z), 1+|z|^2
\right), \ \  |z|<1.
$$
\subsection{}
 Minimal surfaces in $\Nil$ are characterized in terms of 
 the normal Gauss map as follows.
 \begin{Theorem}\label{thm:mincharact}
 Let $f : \D \to \Nil$ be a conformal immersion which is 
 nowhere vertical and $\alpha^{\lambda}$
 the $1$-form defined in \eqref{eq:alpha}.
 Moreover, assume that the unit normal $f^{-1} N$ 
 defined in \eqref{eq:Nell} is upward.
 Then the following statements are equivalent{\rm:}
 \begin{enumerate}
 \item $f$ is a minimal surface.
 \item $d + \alpha^{\lambda}$ is a family of flat connections on $\D \times  \ISU$.
 \item The normal Gauss map $g$ for $f$ is a non-conformal harmonic 
       map into the hyperbolic $2$-space $\mathbb H^2$.
 \end{enumerate}
 \end{Theorem}
\begin{proof}
 The equivalence of $(1)$ and $(2)$ follows immediately from
 Theorem \ref{thm:CMCcharact} in view of Lemma \ref{thm:min-sym}.

 Next we consider $(2) \Rightarrow (3)$.
 Since $\alpha^{\lambda}$ takes values in $\isu$, there exists
 a solution matrix to \eqref{eq:Lax-Nil} which is contained in $\ISU$.
 We express this matrix in terms of spinors $\psi_1$ and $\psi_2$.
 First we recall that the vector $\psi =  e^{-w/4} ( \psi_1, \psi_2)$
 solves this equation. Since in our case now $e^{w/2}$ is purely imaginary, 
 it is straightforward to verify that also the vector 
 $\psi^* = \overline{e^{-w/4}} ( \overline{\psi_2}, \overline{\psi_1})$ solves 
 the same system of differential equations.

 Let $S$ be the fundamental system of solutions to the 
 system \eqref{eq:Lax-Nil} which has the vector $\psi$ as its 
 first row and the vector $\psi^*$ as its second row.
 Since the coefficient matrices have trace $=0$, 
 we know $\det S =\mathrm{constant}$.
 From the form of $S$ we infer 
 $\det S = |e^{-w/4}|^2( |\psi_1|^2 - |\psi_2|^2)$.

 By assumption, the normal is ``upward'', whence $\det S >0$, 
 see \eqref{eq:Nell}.
 As a consequence, after multiplying $S$ by some positive real constant $c$
 (actually, $ c = 1/\sqrt{2}$)
 we can assume $\det (cS) = 1.$ Taking into account the form of 
 $S$ and $cS$ we see that $cS$ is a solution to  \eqref{eq:Lax-Nil}
 which takes values in $\ISU$.

 Let $F$ be a family of maps such that 
 $F^{-1} d F = \alpha^{\lambda}$ with $F|_{\lambda =1} = c S$
 and define a map $N_m$ by
 $$
 N_m = \frac{i}{2}\ad (F) \sigma_3|_{\lambda=1}.
 $$
 Clearly, $N_m$ takes values in $\mathbb{H}^2\subset \isu$. 
 Let $\isu=\mathfrak{u}_{1} \oplus \mathfrak{p}$ denote 
 the Cartan decomposition of the Lie 
 algebra $\isu$
 induced by the derivative of $\sigma =\ad (\sigma_3)$. 
 Here the linear subspace 
 $\mathfrak{p}$ is identified with the tangent space of 
 $\mathbb{H}^2$ at the origin $i\sigma_{3}/2$. 
 It is known, \cite[Proposition 3.3]{DPW} and Appendix 
 \ref{sc:harmonic-homogeneous},
 that $N_m$ is harmonic if and only if 
 \begin{equation}\label{eq:harmonicity}
 F^{-1} d F = \lambda^{-1} \alpha_1^{\prime} + \alpha_0+
 \lambda \alpha_1^{\prime \prime},
 \end{equation}
 where $\alpha_0: T\D \to \mathfrak{u}_1$ and 
 $\alpha_1: T\D \to \mathfrak{p}$ are $\mathfrak{u}_1$ 
 and $\mathfrak p$ valued $1$-forms respectively,
 and superscripts $\prime$ and $\prime \prime$ denote 
 $(1, 0)$ and $(0, 1)$-part respectively. It is easy to check that 
 $\alpha^{\lambda}$, as defined in \eqref{eq:alpha2} 
 coincides with the right hand side of \eqref{eq:harmonicity}.

 In terms of the generating spinors $\psi_1$ and $\psi_2$, 
 the map $N_m$ can be computed as
 $$
 N_m =  \frac{i}{2}\ad (F) \sigma_3|_{\lambda=1}
 = \frac{i}{2(|\psi_1|^2 - |\psi_2|^2)} 
\begin{pmatrix} |\psi_1|^2 + |\psi_2|^2 & 2 i \psi_1 \psi_2  \\
 2 i \overline{\psi_1} \overline{\psi_2}  & -|\psi_1|^2 - |\psi_2|^2
\end{pmatrix},
 $$
 where we set
 \begin{equation}\label{eq:F}
 F|_{\lambda =1} =\frac{1}{\sqrt{|\psi_1|^2-|\psi_2|^2}} 
 \begin{pmatrix} \sqrt{i}^{-1} \psi_1 & \sqrt{i}^{-1}\psi_2 \\ 
 \sqrt{i} \;\overline{\psi_2} & \sqrt{i} \;\overline{\psi_1} 
 \end{pmatrix}.
\end{equation}
 Applying the stereographic projection
 $\pi_h:\mathbb{H}^2 \subset  \mathbb E^{2,1}\to \mathcal{D} \subset \C$ 
 as in \eqref{eq:Minpro}
 with base point $-i\sigma_{3}/2$ to $N_m$, we obtain
 $$
 \pi_h \circ  N_m = \frac{\psi_2}{\overline{\psi_1}}.
 $$
 As a consequence, the map $\pi_g \circ  N_m$ is actually the normal Gauss map 
 $g$ given in \eqref{eq:Normal} and we have $|g|<1$, 
 since we assumed that $f$ is nowhere vertical and $f^{-1} N$ upward.
 Moreover, $g$ can be considered as a harmonic map into $\mathbb{H}^2$ 
 through the 
 stereographic projection. 
 Since the $(1,0)$-part of the upper right entry of  
 $\alpha^{\lambda}|_{\lambda =1}$ is non-degenerate, the normal Gauss map 
 $g$ is non-conformal. 
 Therefore, $(2)$ implies $(3)$.

 Finally, we consider $(3) \Rightarrow (2)$. By assumption we know that the
 normal Gauss map $g$ is harmonic. 
 Therefore a loop group approach is applicable \cite{DPW}.
 In particular, there is a moving frame $F$ which takes values in $\ISU$ 
 from which $g$ can be obtained by projection 
 to $\ISU / \mathrm{U}_1=\mathbb{H}^2$.  But now the result proven 
 in the next section can be applied and the claim is proven.
\end{proof}
\begin{Remark}
\mbox{}
\begin{enumerate}
\item In the theorem above we have made two additional assumptions:
 ``nowhere vertical'', which means no branch points and ``upward''.
  The first condition is also equivalent with $|\psi_1| \neq |\psi_2|$. 
  Hence the Gauss map does not reach the boundary of $\mathbb H^2$.
  The second condition implies that the Gauss map always stays 
  inside the unit disk, that is, the upper hemisphere of $\mathbb S^2$
  and never move across the unit circle to the lower hemisphere 
  of $\mathbb S^2$.

\item The harmonicity of the normal Gauss map $g$ for a minimal surface $f$ 
 can be seen from the partial differential equation for 
 $g$, see \cite{Daniel:GaussHeisenberg, Ino:MiniHeisenberg}.

\item The minimal surface corresponding to a normal Gauss map $g$ 
 is uniquely determined as follows: Let $f$ and $\hat f$ be two minimal surfaces 
 with the same Gauss map $g$. Then the moving frames $F$ and $\hat F$
 to $f$ and $\hat f$ are in the relation $F =\hat F K_0$ for 
 some $K_0 \in \mathrm{U}_1$, and a straightforward computation shows that 
 $K_0$ is constant. It is easy to see that the respective generating 
 spinors $\psi_i$ and $\hat \psi_i\; (i =1, 2)$ for $f$ and $\hat f$ satisfy  $\psi_i = k_0 \hat \psi_i$ for some constant $k_0$, 
 which means that $\psi_i$ and $\hat \psi_i$ are the same up to a change of coordinates.
\end{enumerate}
\end{Remark}

\begin{Definition}
 Let $f$ be a minimal surface in $\Nil$ and 
 $F$ as above 
 the corresponding $\ISU$-valued solution to the equation
 $F^{-1} d F = \alpha^{\lambda}, \lambda \in \mathbb S^1$, 
 where $\alpha^{\lambda}$ is defined by \eqref{eq:alpha} and 
 $F|_{\lambda =1}$ is given in \eqref{eq:F}. Then $F$ is called 
 \textit{extended frame} of the minimal surface $f$.
\end{Definition}
 For later reference we express the extended frame 
 associated with respect to the generating spinors 
 $\psi_1$ and $\psi_2$ for a minimal surface;
 \begin{equation}\label{eq:extframin}
 F(\lambda) =\frac{1}{\sqrt{|\psi_1(\lambda)|^2-|\psi_2(\lambda)|^2}} 
 \begin{pmatrix}
 \sqrt{i}^{-1} \psi_1(\lambda) & 
 \sqrt{i}^{-1} \psi_2(\lambda) \\ 
 \sqrt{i} \;\overline{\psi_2(\lambda)} & 
  \sqrt{i}\; \overline{\psi_1(\lambda)}
 \end{pmatrix}. 
 \end{equation}
 We would like to note that the functions $\psi_1(\lambda)$ 
 and $\psi_2(\lambda)$ in this expression are only 
 determined up to some positive real function.
\section{Sym formula}\label{sc:Sym}
 In this section, we present an immersion formula for minimal 
 surfaces in $\Nil$.
 This formula will be called the \textit{Sym-formula}. It involves
 exclusively the 
 extended frames of minimal surfaces. We will also explain the relation 
 to another formula for $f$ stated in \cite{Cartier}.
\subsection{}
 We first identify the Lie algebra 
 $\mathfrak{nil}_3$ of $\Nil$ with the 
 Lie algebra $\isu$ as a \textit{real vector space}. 
 In $\isu$, we choose the following basis:
\begin{equation}\label{eq:basis}
 \mathcal{E}_1 = \frac{1}{2} \begin{pmatrix} 0 & i \\ -i &0 \end{pmatrix}, \;\;
 \mathcal{E}_2 = \frac{1}{2} \begin{pmatrix} 0 & -1 \\ -1 & 0 \end{pmatrix}\;\;
 \mbox{and}\;\;\;
 \mathcal{E}_3 = \frac{1}{2} \begin{pmatrix} -i & 0\\ 0 &i \end{pmatrix}.
\end{equation}
 One can see that $\{\mathcal{E}_1,\mathcal{E}_2,\mathcal{E}_3\}$ is an 
 orthogonal basis 
 of $\isu$ with timelike vector 
 $\mathcal{E}_3$. A linear isomorphism $\Xi:\mathfrak{su}_{1,1}\to 
 \mathfrak{nil}_3$ is then given by
 \begin{equation}\label{eq:Nilidenti}
\mathfrak{su}_{1,1} \ni 
x_1 \mathcal{E}_1 + x_2 \mathcal{E}_2 + x_3 \mathcal{E}_3
\longmapsto
x_1 e_1 +  x_2 e_2 +  x_3 e_3 \in \mathfrak{nil}_{3}.
\end{equation}
 Note that the linear isomorphism $\Xi$ is not a Lie algebra 
 isomorphism. 
 Next we consider the exponential map 
 $\exp:\mathfrak{nil}_3\to \mathrm{Nil}_3$ defined in \eqref{exp-map}.
 We define a smooth bijection  
 $\Xi_{\rm nil}:\isu \to \Nil$ by $\Xi_{\rm nil}:=\exp \circ \Xi$.
 In what follows we will take derivatives 
 for functions of $\lambda$.
 Note that for $\lambda=e^{i\theta} \in \mathbb S^1$, we have
 $\partial_{\theta}=i \lambda \partial_{\lambda}$.
\begin{Theorem}\label{thm:Sym}
 Let $F$ be the extended frame
 for some minimal surface, 
 $m$ and $N_m$ respectively the maps 
 \begin{equation}\label{eq:SymMin}
 m=-i \lambda (\partial_{\lambda} F) F^{-1} 
 -N_m\;\;
 \mbox{and} \;\;
 N_m= \frac{i}{2} \ad (F) \sigma_3.
 \end{equation}
 Moreover, define a map  $f^{\lambda}:\mathbb{D}\to \mathrm{Nil}_3$ by
 $f^{\lambda}:=\Xi_{\mathrm{nil}}\circ \hat{f^{\lambda}}$ with
\begin{equation}\label{eq:symNil}
 \hat f^{\lambda} = 
    \left.
    \left(m^o -\frac{i}{2} \lambda (\partial_{\lambda} m)^d\right)
    \;\right|_{\lambda \in \mathbb{S}^1}, 
\end{equation}
 where the superscripts ``$o$'' and ``$d$'' denote the off-diagonal and 
 diagonal part, 
 respectively. Then, for each $\lambda \in \mathbb{S}^1$, 
 the map $f^{\lambda}$ is a minimal surface in $\Nil$ and 
 $N_m$ is the normal Gauss map of $f^{\lambda}$. 
 In particular, $f^{\lambda}|_{\lambda =1}$ gives 
 the original minimal surface up to a rigid motion.
\end{Theorem}
\begin{proof}
 Since $m$ and $i \lambda (\partial_{\lambda} m^{d})$ take values
 in the Lie algebra of $\ISU$, the map $f^{\lambda}$ takes values in $\Nil$
 via the bijection  $\Xi_{\rm nil}$.
 Let us express the extended frame $F(\lambda)$  by $\psi_{1}(\lambda)$ and 
 $\psi_{2}(\lambda)$ as in \eqref{eq:extframin}.
 We note that $\psi_1(\lambda)$ and $\psi_2(\lambda)$ depend on $\lambda$
 and for each $\lambda \in  \mathbb S^1$, the extended frame 
 $F$ takes values in $\ISU$.
 Then a straightforward computation shows that 
\begin{eqnarray}\label{eq:derivative}
 \partial_z m &=& \ad (F)
 \left(-i \lambda \partial_{\lambda} U^{\lambda}-\frac{i}{2}[U^{\lambda}, \sigma_3] \right) \\ \nonumber &=&
 -2 i\lambda^{-1} e^{w/2} \ad (F)
 \begin{pmatrix} 
 0 & 1 \\ 0 & 0
 \end{pmatrix}  \\ 
\nonumber
&=& 
\nonumber
\lambda^{-1}
\begin{pmatrix}
 - \psi_1(\lambda) \overline{\psi_2(\lambda)} & - i \psi_1(\lambda)^2 \\
 -i \overline{\psi_2(\lambda)}^2 & \psi_1(\lambda) \overline{\psi_2(\lambda)}
\end{pmatrix}.
\end{eqnarray}
 Thus 
\begin{equation}\label{eq:phiequation}
\partial_z m= 
\phi_{1}(\lambda) \mathcal{E}_1 + \phi_{2}(\lambda) \mathcal{E}_2
-i \phi_3(\lambda) \mathcal E_3
\end{equation}
 with
\begin{equation*}
 \phi_{1}(\lambda) = \lambda^{-1}\left(\overline{\psi_2(\lambda)}^2 - \psi_1(\lambda)^2\right),\;\;
 \phi_{2}(\lambda) = i\lambda^{-1}\left(\overline{\psi_2(\lambda)}^2 + \psi_1(\lambda)^2\right)
\end{equation*}
and
\begin{equation*}
\phi_3(\lambda) = 2\lambda^{-1} \psi_1(\lambda) \overline{\psi_2(\lambda)}.
 \end{equation*}
 Thus using \eqref{eq:derivative}, 
 the derivative of $m$ with respect to $z$ and $\lambda$ can be 
 computed as
\begin{eqnarray}\label{eq:derivative2}
 \partial_z ( i \lambda (\partial_{\lambda} m)) = 
 i \lambda \partial_{\lambda} (\partial_z m) & = &
 i \lambda \partial_{\lambda} \left(-2 i \lambda^{-1} e^{w/2} \ad (F) 
 \begin{pmatrix} 
 0 & 1 \\ 0 & 0
 \end{pmatrix}
 \right), \\
 &=& -i(\partial_z m)
   -\left[m + N_m, \partial_z m\right].
\nonumber
\end{eqnarray}
 Here $[a, b]$ denotes the 
 usual bracket of matrices, that is, $[a, b] = ab -b a$.
 Using \eqref{eq:derivative}, we have 
\begin{equation*}
 \left[-N_m, \partial_z m \right]^d 
 = - i (\partial_z m)^d 
\end{equation*}
 and
\begin{equation*}
 - [m, \partial_z m]^d = \left(\phi_1(\lambda) 
 \int \phi_2(\lambda) \, dz - \phi_2(\lambda) \int \phi_1(\lambda) \, dz \right) 
 \mathcal{E}_3.
\end{equation*}
 Thus we have
\begin{equation}\label{eq:partialfm}
\partial_z \left( -\frac{i \lambda \partial_{\lambda} m}{2}^d\right) 
 = \left(
  \phi_3(\lambda) -\frac{1}{2} \phi_1(\lambda) \int \phi_2(\lambda)\, dz  
  + \frac{1}{2}\phi_2(\lambda) \int \phi_1(\lambda)\, dz 
  \right) 
 \mathcal{E}_3.
\end{equation}
 Therefore, combining \eqref{eq:phiequation} and \eqref{eq:partialfm}, we obtain
$$
  \partial_z \hat f^{\lambda} = \phi_1(\lambda) \mathcal{E}_1 + \phi_2(\lambda) \mathcal{E}_2 
  +\left(\phi_3(\lambda)  - \frac{1}{2}\phi_1(\lambda) \int \phi_2(\lambda)\, dz
  + \frac{1}{2}\phi_2(\lambda) \int \phi_1(\lambda)\, dz
 \right) \mathcal{E}_3.
$$
 We now use the identification \eqref{eq:Nilidenti} with the left translation 
 $(f^{\lambda})^{-1}$, that is, 
 \begin{equation}\label{eq:immersion}
 (f^{\lambda})^{-1} \partial_z f^{\lambda} =  \phi_1(\lambda) e_1 + 
 \phi_2(\lambda) e_2 + \phi_3(\lambda) e_3.
\end{equation}
 Thus $\lambda^{-1/2}\psi_1(\lambda)$ and $\lambda^{1/2}\psi_2(\lambda)$ 
 are spinors for $f^{\lambda}$ for 
 each $\lambda \in \mathbb{S}^1$. 
 In particular, the function
 $$
 \frac{i}{2}(|\lambda^{-1/2}\psi_1(\lambda)|^2 
 - |\lambda^{1/2}\psi_2(\lambda)|^2)
 = e^{w/2}
 $$ 
 does not 
 depend on $\lambda$ and implies that 
 the mean curvature $H$ is equal to zero. 
 Moreover, the conformal factor of the induced 
 metric of $f^{\lambda}$ is given by  
 $$
 e^{u} = 4(|\psi_1(\lambda)|^2 + |\psi_2(\lambda)|^2)^2.
 $$
 This metric is non-degenerate, since $F$ takes values in  $\ISU$ 
 for each $\lambda \in \mathbb S^1$, 
 that is, $|\psi_1(\lambda)|$ and $|\psi_2(\lambda)|$ 
 are not simultaneously equal to zero.
 Thus the map $f^{\lambda}$ actually defines a minimal surface 
 in $\Nil$ for each $\lambda \in \mathbb{S}^1$. 
 By the same argument as in the proof of Theorem \ref{thm:CMCcharact}
 for the spinors  $\lambda^{-1/2}\psi_1(\lambda)$ 
 and $\lambda^{1/2}\psi_2(\lambda)$, 
 the map $N_m$ is the normal Gauss map for the minimal surface $f^{\lambda}$. 
 Then, at $\lambda =1$, the minimal surface given by $f^{\lambda}|_{\lambda=1}$ 
 and the original minimal surface have the same metric $e^u dz d\bar z$, the 
 holomorphic differential $Bdz^2$ and the support $h \sdz \sdzb$.
 Thus up to a rigid motion it is the same minimal surface. 
 This completes the proof.
\end{proof}
\begin{Remark}
\mbox{}
\begin{enumerate}
\item For each $\lambda \in \mathbb S^1$ the immersion $m$
      defined in \eqref{eq:SymMin} gives a spacelike surface of 
      constant mean curvature in Minkowski $3$-space 
      $\mathbb E^{2,1} =\mathfrak{su}_{1,1}$, see \cite{Kob:Realforms, BRS:Min}.
      It is well known that the Sym formula 
      for constant mean curvature 
      surfaces in $\mathbb E^{2,1}$(or $\mathbb E^3$) involves the first 
      derivative 
      with respect to $\lambda$ only, however, the formula for $\Nil$
      involves the second derivative with respect to $\lambda$ as well.
      Purely technically the reason is the subtraction term.
      But there should be a better geometric reason.

 \item  Theorem \ref{thm:Sym} gives clear geometric meaning for the immersion formula for $f$. 
  The Sym formula \eqref{eq:symNil} for $f$ 
  was written down in \cite{Cartier} in a different way.
\end{enumerate}
\end{Remark}

 In the following Corollary, we compute 
 the Abresch-Rosenberg differential $4 B^{\lambda} dz^2$ for 
 the $1$-parameter family $f^{\lambda}$ in Theorem \ref{thm:Sym}
 and it implies that the family $f^{\lambda}$ actually 
 defines the associated family.
\begin{Corollary}\label{coro:associate}
 Let $f$ be a conformal minimal surface in $\Nil$ and 
 $f^{\lambda}$  the family of surfaces defined by \eqref{eq:symNil}.
 Then $f^{\lambda}$ preserves 
 the mean curvature $(=0)$ and the support. The
 Abresch-Rosenberg differential $4B^{\lambda}dz^2$ for $f^{\lambda}$
 is given by $4B^{\lambda} dz^2= 4 \lambda^{-2} Bdz^2$, 
 where $4 Bdz^2$ is the 
 Abresch-Rosenberg differential for $f$. Therefore 
 $\{f^{\lambda}\}_{\lambda \in \mathbb S^1}$
 is the associated family of the minimal surface $f$.
\end{Corollary}
\begin{table}[t]
\extrarowheight=1mm
\begin{tabular}{|c|c|c|c|c|}\hline
 {\small surface} &{\small mean curvature} & 
 {\small  metric} & {\small holo. differential} 
 &{\small support} \\[1mm]\hline
 $f(=f^{\lambda}|_{\lambda =1})$ & $H =0$ & $\exp (u)dz d\bar z$ & $B dz^2$ & $h \sdz
 \sdzb$ \\[1mm]\hline
 $f^{\lambda}$ &$H =0$ &$\exp (u^{\lambda}) dz d\bar z$ & $\lambda^{-2} B dz^2$ &$h \sdz \sdzb$\\[1mm]\hline
\end{tabular}
\vspace{0.3cm}
 \caption{An original minimal surface $f$ and the deformation family $f^{\lambda}$.}
\end{table}
\begin{proof}
 From Theorem \ref{thm:CMCcharact} it is clear that 
 a minimal surface $f$ in $\Nil$ defines a $1$-parameter 
 family $f^{\lambda}$
 of the minimal immersion $f$ such that $f^{\lambda}|_{\lambda=1} =f$, 
 and $f^{\lambda}$ is 
 a minimal surface for each $\lambda \in \mathbb{S}^1$. 
 The spinors for $f^{\lambda}$ are given as the functions 
 $\lambda^{-1/2} \psi_1(\lambda)$ and $\lambda^{1/2} \psi_2(\lambda)$, 
 see the proof of Theorem \ref{thm:CMCcharact}. 
 Using \eqref{eq:alpha2} and \eqref{eq:extframin} we obtain
 \begin{equation}\label{eq:associated}
 (\lambda^{-1/2} \psi_1(\lambda))_z 
 =  \frac{1}{2}w_z(\lambda^{-1/2} \psi_1(\lambda)) 
   + (\lambda^{1/2}\psi_2(\lambda))
 (\lambda^{-2} B) e^{-w/2}.
\end{equation}
 Comparing \eqref{eq:associated} to \eqref{eq:Lax-Niltilde}, 
 the Abresch-Rosenberg differential $4B^{\lambda}dz^2$ for $f^{\lambda}$ 
 is $4\lambda^{-2} Bdz^2$. Moreover $B^{\lambda}$ 
 is holomorphic, since $B$ is holomorphic. 
 We note that the support function for $f^{\lambda}$
 is given by $e^{w/2} = i(|\psi_1(\lambda)|^2-|\psi_2(\lambda)|^2)/2$,
 which is invariant in this family. 
 Therefore this $1$-parameter family $f^{\lambda}$ is 
 the associated family as explained in Section \ref{subsc:associated}.
\end{proof}
\begin{Remark}
 In general, the metric $e^{u} dz d \bar z 
  = 4(|\psi_1(\lambda)|^2+|\psi_2(\lambda)|^2)^2dz d\bar z$ is 
  not preserved in the associated family. 
  This is in contrast to the case of an associated family of
  nonzero constant mean curvature surfaces in $\mathbb E^3$ or 
  $\mathbb E^{2,1}$, where the metric is preserved.
\end{Remark}
\section{Potentials for minimal surfaces}\label{sc:potential}
 In this section, we show that 
 pairs of meromorphic and anti-meromorphic $1$-forms, 
 the so-called \textit{normalized potentials}, are obtained
 from the extended frames of minimal surfaces 
 in $\Nil$ via the Birkhoff decomposition of loop groups.

 We first define the twisted $\SL$ loop group as a space of continuous maps 
 from $\mathbb{S}^1$ to the Lie group $\SL$, that is, 
 $$
 \LSL =\{g : \mathbb{S}^1 \to \SL \;|\; g(-\lambda) = \sigma g(\lambda) \},
 $$
 where $\sigma =\ad (\sigma_3)$.  We restrict 
 our attention to loops in $\LSL$ such that 
 the associate Fourier series of the loops are absolutely convergent.
 Such loops determine a Banach algebra, the so-called  
 \textit{Wiener algebra}, and 
 it induces a topology on $\LSL$, 
 the so-called  \textit{Wiener topology}.
 From now on, we consider only $\LSL$ equipped with the Wiener topology.

 Let $D^{\pm}$ denote respective the inside of unit disk and
 the union of outside of the unit disk and infinity.
 We define \textit{plus} and \textit{minus} loop subgroups of $\LSL$;
\begin{equation}
\LSLPM=\{ g \in \LSL \;|\; \mbox{$g$ can be extended holomorphically to $D^{\pm}$} \}.
\end{equation}
 By $\LSLPI$ we denote the subgroup of 
 elements of $\LSLP$ which take the value identity at zero.
 Similarly, by $\LSLMI$ we denote the subgroup of 
 elements of $\LSLM$ which take the value identity at infinity.
 
 We also define the $\ISU$-loop group as follows:  
 \begin{equation}\label{eq:SU-loop}
 \LISU =\left\{ g  \in \LSL  \;|\;  \sigma_3 \overline{g(1/\bar \lambda)}^{t-1} \sigma_3 
 = g(\lambda)\right\}.
 \end{equation}
 It is clear that extended frames of minimal surfaces in $\Nil$ are 
 elements in $\LISU$.
\begin{Theorem}[Birkhoff decomposition, \cite{PreS:LoopGroup}]\label{thm:Birkhoff}
 The respective multiplication maps
 \begin{equation}
 \LSLMI \times \LSLP \to \LSL \;\;\mbox{and} \;\;\LSLPI \times \LSLM \to \LSL
 \end{equation}
 are analytic diffeomorphisms onto open dense subsets of $\LSL$.
\end{Theorem}
 It is easy to check that the extended frames $F$ are elements in $\LSL$, 
 since $U^{\lambda}$ and $V^{\lambda}$ satisfy the twisted condition.
 Applying the Birkhoff decomposition of Theorem \ref{thm:Birkhoff} to 
 the extended frame $F$, we obtain a pair of meromorphic 
 and anti-meromorphic $1$-forms, that is, the pair of normalized potentials.

\begin{Theorem}[Pairs of normalized potentials]\label{thm:Normalized}
 Let $F$ be the extended frame of some minimal 
 immersion in $\Nil$ on some simply connected domain $\D \subset \C$ 
 and decompose $F$ as $F = F_{-} V_{+}  = F_{+} V_{-}$ 
 according to Theorem \ref{thm:Birkhoff}. 
 Then $F_{-}$ and $F_{+}$ are meromorphic and 
 anti-meromorphic respectively. 
 Moreover, the Maurer-Cartan forms 
 $\xi_{\pm}$ of $F_{\pm}$ are 
 given explicitly as follows:
\begin{equation}\label{eq:pair}
 \left\{
 \begin{array}{l}
 \xi_{-}(z, \lambda) = F_{-}^{-1}(z, \lambda) d F_{-}(z, \lambda) 
      = \lambda^{-1} \begin{pmatrix} 0 & -p \\ B p^{-1} & 0 \end{pmatrix} dz, \\
 \xi_{+}(z, \lambda) = F_{+}^{-1}(z, \lambda) d F_{+}(z, \lambda) =
 -\sigma_3 \overline{\xi^c_{-}(z, 1/\bar \lambda)}^t \sigma_3 d\bar z,
 \end{array}
\right.
\end{equation}
 where $p$ is a meromorphic function on $\D$,  
 $\xi^c_{-}(z, \lambda)$ denotes the coefficient matrix of $\xi_{-}(z, \lambda)$ 
 and $B$ is the holomorphic function on $\D$ 
 defined in \eqref{eq:ARdiff}, 
 which is the coefficient of the Abresch-Rosenberg differential.
\end{Theorem}
 \begin{proof}
 From the equality $F = F_{-} V_+$, 
 the Maurer-Cartan form of $F_{-}$ can be computed as 
 $$
 \xi_{-} = F_{-}^{-1} d F_{-} =  V_{+} F^{-1} (d F V_+^{-1}
 - F V_{+}^{-1} dV_{+} V_{+}^{-1}) = \ad (V_+) \alpha^{\lambda} 
 - dV_{+} V_{+}^{-1}.
 $$
 Since the coefficient matrix of $\xi_{-}$ 
 is an element in the Lie algebra of $\LSLMI$ and 
 the lowest degree of entries of 
 the right hand side with respect to $\lambda$ 
 is equal to $-1$, the $1$-form
 $\xi_{-}$ can be computed as 
 $$
 \xi_- = \lambda^{-1}
 \begin{pmatrix} 0 &  - e^{w/2} v_+^{2} \\ B e^{-w/2}v_+^{-2} & 0 \end{pmatrix} dz,
 $$
 where $\di(v_+, v_+^{-1})$ is the constant coefficient of the Fourier 
 expansion of $V_+$ with respect to $\lambda$. Moreover 
 from \cite[Lemma 2.6]{DPW}, it is known that $F_-$ is 
 meromorphic on $\D$, and thus $\xi_{-}$ is meromorphic on $\D$.
 Setting $p = e^{u/2} v_+^2$, we obtain the form $\xi_-$ in \eqref{eq:pair}.
 Similarly, by the equality $F= F_+ V_-$, 
 the $1$-form $\xi_+$ can be computed as 
 $$
 \xi_{+} = F_{+}^{-1} d F_{+} =  V_{-} F^{-1} (d F V_-^{-1}
 - F V_{-}^{-1} dV_{-} V_{-}^{-1}) = \ad (V_-) \alpha^{\lambda} 
 - dV_{-} V_{-}^{-1}.
 $$
 Since the coefficient matrix of $\xi_+$ is an element 
 in the Lie algebra of $\LSLP$ and 
 the highest degree of entries of 
 the right hand side with respect to $\lambda$ 
 is equal to $1$, 
 the $1$-form $\xi_+$ can be computed as 
 $$
 \xi_+ = \lambda
 \begin{pmatrix} 0 &  - \bar B e^{-w/2} v_-^{2} \\  e^{w/2}v_-^{-2} & 0\end{pmatrix} 
 d\bar z,
 $$
 where $\di(v_-, v_-^{-1})$ is the constant coefficient of the Fourier 
 expansion of $V_-$ with respect to $\lambda$. Similar to the case of 
 $\xi_-$, from \cite[Lemma 2.6]{DPW} it is known that $F_+$ is 
 anti-meromorphic on $\D$, and thus $\xi_{+}$ is anti-meromorphic on $\D$.
 Since $F$ has the symmetry $F(\lambda)
 = \sigma_3 \overline{F(1/ \bar \lambda)}^{t -1} \sigma_3$, 
 $$
 F(\lambda) = \sigma_3 \overline{F_{-}(1/ \bar \lambda)}^{t -1} \sigma_3
 \sigma_3 \overline{V_+(1/ \bar \lambda)}^{t -1} \sigma_3
 $$
 is the second case in the Birkhoff decomposition 
 of Theorem \ref{thm:Birkhoff}. Since the Birkhoff decomposition 
 is unique, $F_+$ can be computed as 
 $$
 F_+(\lambda) = \sigma_3 \overline{F_-(1/\bar \lambda)}^{t -1} \sigma_3.
 $$
 Therefore, $\xi_+ = F_+^{-1} dF_+$ has the symmetry as stated 
 in \eqref{eq:pair}.
 \end{proof}
 \begin{Definition}
  The pair of meromorphic and anti-meromorphic $1$-forms $\xi_{\pm}$ 
  defined in \eqref{eq:pair} is called the 
  {\it pair of normalized potentials}.
 \end{Definition}

\section{Generalized Weierstrass type representation for minimal surfaces in  $\Nil$}\label{sc:Weierstrass}
 In the previous section, a pair of normalized 
 potentials was obtained from a minimal surface in $\Nil$. 
 In this section, we will conversely show the generalized Weierstrass type 
 representation formula for minimal
 surfaces in $\Nil$ from pairs of normalized potentials.

 {\bf Step I.} Let $(\xi_{-}, \xi_+)$ be a pair of normalized potentials 
 defined in \eqref{eq:pair}. Solve the pair of ordinary differential equations:
\begin{equation}\label{eq:solCpm}
 d C_{\pm} = C_{\pm} \xi_{\pm}, 
\end{equation}
 where $C_{+}(\bar z_*, \lambda) = 
 \sigma_3 \overline{C_{-}(z_*, 1/\bar \lambda)}^{t-1}
 \sigma_3$ and the initial condition $C_{-}(z_*, \lambda)$ is 
 chosen such that 
 $C_{-}^{-1}(z_*, \lambda) C_{+}(\bar z_*, \lambda)$ is 
 Birkhoff decomposable in both ways of Theorem \ref{thm:Birkhoff}.

 {\bf Step I\!I.} Applying the Birkhoff decomposition of 
 Theorem \ref{thm:Birkhoff}
 to the element $C_{-}^{-1} C_{+}$, we obtain 
 for almost all $z$
 \begin{equation}\label{eq:FBirkhoff}
 C_{-}^{-1} C_{+} = V_{+} V_{-}^{-1}, 
 \end{equation}
 where $V_{+} \in \LSLPI$ and $V_{-} \in  \LSLM$.
\begin{Remark}
\mbox{}
\begin{enumerate}
\item  If we change the initial condition of $C_{-}$ from 
 $C_{-}(z_*, \lambda)$ to $U(\lambda)C_{-}(z_*, \lambda)$ 
 by an element $U(\lambda)$ in $\LISU$ which is independent of $z$, 
 then the Birkhoff decomposition for 
 $\tilde C_{-}^{-1} \tilde C_+$ with $\tilde C_{-}(z, \lambda) = U(\lambda) C_-(z, \lambda)$ 
 and  $\tilde C_{+}(z, \lambda) = \sigma_3 \overline{U(1/\bar \lambda)}^{t -1} \sigma_3 C_+(z, \lambda)$ 
 is the same as $C_{-}^{-1} C_{+}$, that is,
 $$
 \tilde C_{-}^{-1} \tilde C_+ =C_{-}^{-1} C_{+} = V_{+} V_{-}^{-1},
 $$
 since $U(\lambda)$ is in $\LISU$.

\item The Birkhoff decomposition \eqref{eq:FBirkhoff} permits to define 
 $\hat F = C_{-}V_{+}= C_{+} V_{-}$.
 The expressions $C_{-} = \hat F V_{+}^{-1}$ 
 and $C_{+} = \hat F V_{-}^{-1}$
 look like Iwasawa decompositions.
 However, for this we need $\hat F$ to be $\ISU$-valued.
 That one can replace $\hat F$ in some open subset 
 $\mathbb D_0 \subset \mathbb D$ by some $F = \hat F k$, 
 $k$ diagonal, real and independent of $\lambda$ such that $F \in \LISU$, 
 will be shown below. Note however, that such a decomposition can be obtained
 in general only for $z \in \mathbb D_0$, since there are two open Iwasawa cells 
 and if $C_{-}(z, \lambda)$ moves into the second open Iwasawa cell, 
 then $C_{-}(z, \lambda) = \hat F(z, \lambda) \omega_0(\lambda) V_+^{-1}(z, \lambda)$ 
 for some $\omega_0(\lambda)$, see \cite{BRS:Min}.
\end{enumerate}
\end{Remark}
\begin{Theorem}\label{thm:extendedframe}
 Let $F= C_{+} V_{-} = C_{-} V_{+}$ be the loop defined by 
 the Birkhoff decomposition in \eqref{eq:FBirkhoff}. 
 Then $V_{-}|_{\lambda =\infty}$ is a $\lambda$-independent diagonal 
 $\SL$ matrix with real entries. If its real diagonal entries are positive, 
 then there exists a $\lambda$-independent 
 diagonal $\SL$ matrix $D$ such that $F D \in \LISU$ 
 is the extended frame of some minimal surface in $\Nil$ around the base point $z_*$.
 If the real diagonal entries of $V_{-}|_{\lambda =\infty}$ 
 are negative, then there exists a $\lambda$-independent 
 diagonal $\SL$ matrix $D$ and $\omega_0 = 
(\begin{smallmatrix}0 & \lambda \\ -\lambda^{-1}& 0 \end{smallmatrix})$ 
 such that $C_{+} = F V_{-}^{-1} = (F D\omega_0) \omega_0 
 (DV_{-}^{-1})$, where $F D\omega_0 \in \LISU, DV_{-}^{-1} 
 \in \LSLM$ and $\omega_0 F D$ 
 is the extended frame of some minimal surface in $\Nil$ around the base point $z_*$.
\end{Theorem}
\begin{proof}
 From {\bf Step I}, the solution $C_{+}(\bar z, \lambda)$ in \eqref{eq:solCpm} 
 satisfies the symmetry 
 $$
 C_{+} (\bar z, \lambda) =  \sigma_3 \overline{C_-( z, 1/\bar \lambda)}^{t -1} \sigma_3.
 $$
 Therefore 
 \begin{align}\label{eq:V+-}
 V_+ (z,\bar z, \lambda)V_-(z,\bar z, \lambda)^{-1} 
 &= C_{-}(z, \lambda)^{-1} C_{+}(\bar z, \lambda) 
 = \sigma_3 \overline{C_+(\bar z, 1/ \bar \lambda)^{-1} 
 C_-(\bar z, 1/ \bar \lambda)}^{t-1} \sigma_3  \\
 \nonumber
&=\sigma_3 \overline{V_{-}(z,\bar z, 1/ \bar \lambda)}^{t-1}\sigma_3 
  \sigma_3 \overline{V_+ (z,\bar z, 1/ \bar\lambda)}^{t} \sigma_3.
\end{align}
 Thus by the uniqueness of the Birkhoff decomposition of Theorem \ref{thm:Birkhoff},
 we have 
\begin{align}\label{eq:diagonal}
 V_+ (z,\bar z, \lambda) =  \sigma_3 \overline{V_{-}(z,\bar z, 1/ \bar \lambda)}^{t-1}\sigma_3 
 K(z, \bar z),
\end{align}
 where $K$ is some $\lambda$-independent $\SL$ diagonal matrix.
 Let $V_{-}|_{\lambda=\infty} = \di (v_-, v_-^{-1})$ denote the $\lambda$-independent 
 constant coefficient of the Fourier 
 expansion of $V_-$ with respect to $\lambda$. 
 Then by \eqref{eq:V+-} and \eqref{eq:diagonal}, $v_{-}$ takes values in $\mathbb R$ 
 and $K= V_-|_{\lambda=\infty}$.
 Since at $z_* \in \mathbb C$, the loop $C^{-1}_{-} C_{+}$ is Birkhoff decomposable, 
 $v_{-}>0$ or $v_{-}<0$ on some open subset $\mathbb D \subset \mathbb C$ containing 
 $z_*$.
 Let us consider first the case of $v_{-}>0$.
 The Maurer-Cartan form of $F$ can be computed as 
\begin{equation*}
 F^{-1} d F = \ad (V_+^{-1}) \xi_- + V_+^{-1} d V_+ 
                                = \ad (V_-^{-1}) \xi_+ + V_-^{-1} d V_-.  
\end{equation*}
 Since the lowest degree of entries of the middle term is $\lambda^{-1}$
 and the highest degree of entries of the right term is $\lambda$, 
 we obtain
 $$
 F^{-1} d F = 
 \lambda^{-1} \begin{pmatrix} 0  & -p \\ B p^{-1} & 0\end{pmatrix} dz 
 + \alpha_0 +
 \lambda \begin{pmatrix} 0 & \bar B \bar p^{-1} v_-^{-2} \\ - \bar p v_-^2 & 0  \end{pmatrix} 
 d\bar z,
 $$
 where $\alpha_0$ consists of the $dz$-part only and 
 is computed from $V_-^{-1} dV_-$ as 
 $$
 \alpha_0 = 
  \begin{pmatrix}
  v_-^{-1} (v_{-})_z dz  & 0 \\
  0 & - v_-^{-1} (v_{-})_z dz 
  \end{pmatrix}. 
 $$
 Let us consider the change of coordinates $w = \int_{z_*}^z p(t) dt$, 
 that is, $ d w  = p(z) d z$ and $d \bar w  = \overline{p(z)} d \bar z$.
 Then 
 \begin{equation}\label{eq:IwasawaFlambda}
 F^{-1} d F = 
 \lambda^{-1} \begin{pmatrix} 0  & - 1 \\ B p^{-2} & 0\end{pmatrix} dw 
 + \alpha_0 +
 \lambda \begin{pmatrix} 0 & \bar B  {\bar p}^{-2} v_-^{-2} 
 \\ - v_-^2 & 0  \end{pmatrix} 
 d\bar w,
 \end{equation}
 and $\alpha_0$ is unchanged, since $v_-^{-1} (v_{-})_z dz= v_-^{-1} (v_{-})_w dw$.
 Finally choosing the gauge $D = \di (v_-^{-1/2}, v_-^{1/2})$, we have
 \begin{equation}\label{eq:finalMC}
 (FD)^{-1} d(F D) = 
 \lambda^{-1} \begin{pmatrix} 0  & - v_- \\ \tilde B v_-^{-1} & 0\end{pmatrix} 
 dw 
 + \tilde \alpha_0 +
 \lambda \begin{pmatrix} 0 & \bar{\tilde B} v_-^{-1} \\ - v_- & 0  \end{pmatrix} 
 d\bar{w},
 \end{equation}
 with $\tilde B = B p^{-2}$ and 
 $$
 \tilde \alpha_0 = 
 \begin{pmatrix}
 \frac{1}{2} (\log v_-)_{w} dw
 -\frac{1}{2} (\log v_-)_{\bar w} d \bar w
  & 0  \\ 0  &
 - \frac{1}{2} (\log v_-)_{w} d w
 + \frac{1}{2} (\log v_-)_{\bar w} d \bar w 
 \end{pmatrix}. 
 $$
 Thus the Maurer-Cartan form \eqref{eq:finalMC} has the form 
 stated in \eqref{eq:alpha}. Moreover, using \eqref{eq:diagonal}
 and $D^{-2} =K$, we have  
 $$
 V_+ (z,\bar z, \lambda)D(z, \bar z) =  \sigma_3 \overline{V_{-}(z,\bar z, 1/ \bar \lambda) 
 D(z, \bar z)}^{t-1}\sigma_3.
 $$
 Therefore, $F D = C_- V_+ D = C_+ V_{-} D$ takes values in $\LISU$ 
 and is the extended frame for some minimal surface in $\Nil$.
 We now consider the case of $v_{-}<0$. Then similar to the case of $v_{-}>0$, 
 the Maurer-Cartan equation of $F$ is the same as in \eqref{eq:IwasawaFlambda}. Let $\omega_0 =
 (\begin{smallmatrix} 0 & \lambda \\  -\lambda^{-1} & 0 
 \end{smallmatrix})$.
 Then choosing the gauge $D \omega_0
 =\left( \begin{smallmatrix} 0 & \lambda |v_{-}|^{-1/2} \\
  -\lambda^{-1}|v_{-}|^{1/2} & 0 \end{smallmatrix}\right)$
 we have 
$$
 V_+ (z,\bar z, \lambda) D(z, \bar z) \omega_0(\lambda) =  \sigma_3 \overline{V_{-}(z,\bar z, 1/ \bar \lambda) 
 D(z, \bar z) \omega_0(1/\bar \lambda)}^{t-1}\sigma_3.
$$
 Therefore, $F D \omega_0 = C_- V_+ D \omega_0= C_+ V_{-} D \omega_0$ 
 takes values in $\LISU$. 
 Moreover $\omega_0  FD = \omega_0^{-1} (F D \omega_0) \omega_0$
 also takes values in $\LISU$ and its Maurer-Cartan form has 
 the form stated in \eqref{eq:alpha}.  Thus $\omega_0  FD$
 is the extended frame of some minimal surface in $\Nil$.
\end{proof}
 {\bf Step I\!I\!I.} 
 In this final step, minimal surfaces in $\Nil$ can be obtained by 
 the Sym formula (see Theorem \ref{thm:Sym}) for the extended frame
 $F$:
 Let $m =-i \lambda (\partial_{\lambda} F) F^{-1} 
 - \tfrac{i}{2} F \sigma_3 F^{-1}$ and 
 $\hat f^{\lambda}$ as  
\begin{equation*}
 \hat f^{\lambda} = 
    \left.
    \left(m^o -\frac{i}{2} \lambda (\partial_{\lambda} m)^d\right)
    \;\right|_{\lambda \in \mathbb{S}^1}, 
\end{equation*}
 as in Theorem \ref{thm:Sym}. Then via the identification in \eqref{eq:Nilidenti}, 
 for each $\lambda$, the map 
 $f^{\lambda} = \Xi_{\rm nil} \circ \hat f^{\lambda}$ gives a minimal surface in $\Nil$.

\begin{Remark}
\mbox{}
 The normalized potential $\xi_{-}$ in \eqref{eq:pair}
 generating the harmonic map associated with a minimal surface
 can be explicitly computed from the geometric data, by the 
 so-called {\it Wu's formula} as follows:
\begin{equation}\label{eq:Wu}
 \xi_{-} = \lambda^{-1} 
\begin{pmatrix}
 0 & - e^{\hat w (z) - \hat w(0)/2} \\
 B(z) e^{-\hat w(z)+ \hat w(0)/2 } & 0
\end{pmatrix}
 dz, 
\end{equation}
 where $4 Bdz^2$ is the Abresch-Rosenberg differential and 
 $e^{\hat w(z)}$ is the holomorphic extension of
 $e^{w (z, \bar z)} = - h^2(z, \bar z)/16$ around the base point 
 $z=0$ with the support function $h$. 
 The proof of this formula is analogous to the original 
 proof of Wu's formula for constant mean 
 curvature surfaces in $\mathbb E^3$, see \cite{Wu}.
\end{Remark}

\section{Examples}\label{sc:examples}
 We exhibit some examples of minimal surfaces.
 In our general frame work, if one change the initial 
 condition of $C_-$ from $C_-(z_*, \lambda)$ to 
 $U(\lambda) C_-(z_*, \lambda)$ for some $U(\lambda) \in \LISU$, 
 then the corresponding harmonic maps into $\mathbb H^2$ are 
 isometric. However, the associated minimal surfaces can 
 differ substantially, since isometries of $\mathbb H^2$ 
 do not correspond in general to isometries of $\Nil$.
 Since $\ISU$ is a three-dimensional Lie group, 
 the initial conditions $U(\lambda) \in \LISU$ for each $\lambda \in 
 \mathbb S^1$
 could yield three-dimensional families of non-isometric 
 minimal surfaces.
 However, choosing a $\ISU$-diagonal matrix, which is an isometry of 
 $\Nil$ by rotation, the initial conditions in general give only 
 two-dimensional families of non-isometric minimal surfaces.
\subsection{Horizontal umbrellas}\label{sbsc:plane}
 Let $\xi_{-}$ be the normalized potential defined as
$$ 
\xi_{-} = -\lambda^{-1} 
\begin{pmatrix}
0 & i \\
0 & 0
\end{pmatrix} dz.
 $$
 It is easy to compute the solution to $d C_{-} = C_{-} \xi_{-}$ with 
 $C_{-}(z=0, \lambda) = \di (\sqrt{i}^{-1}, \sqrt{i})$.  It is given by 
 $$
 C_{-} = 
 \begin{pmatrix}
 \sqrt{i}^{-1} & -i \sqrt{i}^{-1} \lambda^{-1} z\\ 
 0 & \sqrt{i}
 \end{pmatrix}.
 $$
 Then the loop group decomposition $C_{-} = F V_{+}$ with $F \in \LISU$ and $V_+ \in \LSLP$ 
 can be computed explicitly:
$$
 F =
 \frac{1}{\sqrt{1-|z|^2}}
 \begin{pmatrix}
 \sqrt{i}^{-1} & - i\sqrt{i}^{-1}\lambda^{-1}z\\
 i\sqrt{i}\lambda{\bar z} & \sqrt{i}
 \end{pmatrix}\;\;\mbox{and}\;\;
 V_{+} = \frac{1}{\sqrt{1-|z|^2}}
\begin{pmatrix}
 1& 0\\
 - i\lambda \bar z
 &1-|z|^2
\end{pmatrix}
 $$
 Hence
 $$
 \hat f^{\lambda} =  \frac{2}{1-|z|^2}
 \begin{pmatrix}
 0 & -i \lambda^{-1} z\\ 
 i \lambda \bar z & 0
 \end{pmatrix}.
 $$
 By the identification \eqref{eq:Nilidenti}, we obtain 
 $$
 f^{\lambda} =- \frac{2}{1-|z|^2}  (\lambda^{-1} z+ \lambda \bar z, 
 \; i (\lambda \bar z- \lambda^{-1} z), \; 0).
 $$
 In this case the associated family consists of different 
 parametrizations of the same \textit{horizontal plane}.
 It is easy to see that the Abresch-Rosenberg differential
 of a horizontal plane is zero. 

 Taking a different $\LISU$-initial condition for the above $C_{-}$,
 the resulting surfaces are non-vertical planes:
 Let $\mathcal{F}(x_1,x_2)=ax_{1}+bx_{2}+c$ be a linear function on 
 the $x_1x_2$-plane. Then the graph of $\mathcal{F}$ 
 is a minimal surface in $\mathrm{Nil}_3$ with negative 
 Gaussian curvature
$$
 K=-\frac{3+2(b-x_1/2)^2+2(a+x_2/2)^2}{4\{1+2(b-x_1/2)^2+(a+x_2/2)^2\}^2}
 <0.
$$ 
 Then the graph of $\mathcal{F}$ is called 
 a \textit{horizontal umbrella}. 

\subsection{Hyperbolic paraboloids}\label{sbsc:parboloid}
 Let $\xi_{-}$ be the normalized potential 
 $$ 
\xi_{-} = -\frac{i}{4}\lambda^{-1} 
\begin{pmatrix}
0 & 1 \\
1 & 0
\end{pmatrix} dz.
 $$
 It is easy to compute the solution $d C_{-} = C_{-} \xi_{-}$ with 
 $C_{-}(z=0, \lambda) = \di (\sqrt{i}^{-1}, \sqrt{i})$, which is given by 
 $$
 C_{-} = 
 \begin{pmatrix}
 \sqrt{i}^{-1}\cosh p &  
 \sqrt{i}^{-1}\sinh p \\ 
 \sqrt{i} \sinh p  &  
 \sqrt{i} \cosh p
 \end{pmatrix},
 $$
 where $p=- i \lambda^{-1} z/4$.
 Then the loop group decomposition $C_{-} = F V_{+}$ with $F \in \LISU$ and 
 $V_+ \in \LSLP$ can be
 computed explicitly:
 $$
 F = 
 \begin{pmatrix}
 \sqrt{i}^{-1}\cosh (p+p^*) &  
\sqrt{i}^{-1} \sinh (p +p^*) \\ 
\sqrt{i} \sinh (p +p^*)  &  
 \sqrt{i} \cosh (p +p^*)
 \end{pmatrix}
\;\;\mbox{and}\;\;
 V_{+}  = 
 \begin{pmatrix}
 \cosh p^* &  
 -\sinh p^*\\ 
 -\sinh p^* &  
 \cosh p^*
 \end{pmatrix},
 $$
 where $p^* = i \lambda \bar z/4$.
 A direct computation shows that 
 $$
 \hat f^{\lambda} =
\frac{1}{2}
 \begin{pmatrix} 
(p -\bar p) \sinh (2 (p+p^*)) & 
\sinh (2  (p+\bar p)) +2 (p- p^*) \\
\sinh (2  (p+\bar p)) - 2(p - p^*)& 
-(p-\bar p) \sinh (2  (p+p^*)) 
 \end{pmatrix}.
 $$
 By the identification \eqref{eq:Nilidenti}, we obtain 
 $$
 f^{\lambda} =(-2i (p-p^*), \;
    - \sinh (2 (p+p^*)), \;
     2 i (p- p^*)\sinh (2 (p+ p^*))).
 $$
 This is an associated family of minimal surfaces in $\Nil$ which actually parametrize 
 the same hyperbolic paraboloid, 
 that is,  $x_3 =x_1 x_2/2$.  It is easy to see that 
 the Abresch-Rosenberg differential of a hyperbolic paraboloid 
 is $4 B^{\lambda} dz^2 = \lambda^{-2}/4dz^2$.
 Note that a hyperbolic paraboloid $x_3=x_1x_2/2$ can be written as
 $$
 f(x_1,x_2)=\exp(x_1e_1)\cdot \exp(x_2e_2).
 $$
 Taking a different $\LISU$-initial condition for the above $C_{-}$,
 the resulting surfaces are the \textit{saddle-type examples} of
 \cite{Abresch-Rosenberg}. 
 They are the special case of the 
 \textit{translational-invariant examples}, see \cite{ILM}.
 The \textit{saddle-type minimal surfaces}
 were discovered by Bekkar \cite{Bekkar}, 
 see also \cite[Part II, Proposition 1.9, Remark 1.10]{IKOS}.
 The saddle-type one was also found by \cite{FMP} 
 as translation invariant minimal surfaces.

\begin{Remark} 
 Let $G$ be a compact semi-simple Lie group equipped with a 
 bi-invariant Riemannian metric.
 Take linearly independent vectors $X$, $Y$ in the Lie algebra. 
 Then the map $f:\mathbb{R}^2\to G$ defined by
 $$
 f(x,y)=\exp(xX)\cdot\exp(yY)
 $$ 
 is a harmonic map. 
 Moreover one can see that $f$ is of finite type ($1$-type in 
 the sense of \cite{BP}.) 
\end{Remark}
\subsection{Helicoids and catenoids}\label{sbsc:helicoids}
 We first note that in place of normalized potentials 
 $\xi_{-} = 
 \lambda^{-1}
 \left(
 \begin{smallmatrix}0 & p \\ q & 0 \end{smallmatrix}
 \right)dz$ with $p, q$ meromorphic functions,
 one can also generate  the same surface by 
 \textit{holomorphic potentials} $\eta$
 of the form
 $$
 \eta = \sum_{i=-1}^{\infty} \lambda^{i} \eta_{i},
 $$
 where $\eta_{2k -1}$ and  $\eta_{2k}$ are respectively 
 off-diagonal and diagonal holomorphic 
 $\mathfrak {sl}_2 \mathbb C$-valued $1$-forms, 
 see \cite{Dorfmeister}.
 
 Let $\eta$ be a holomorphic potential of the form
 \begin{equation}\label{eq:delpot}
\eta =  D dz, \;\;\mbox{with} \;\;
 D =
\begin{pmatrix}
c & a \lambda^{-1} + b \lambda \\
 -a \lambda -b \lambda^{-1}  & -c
\end{pmatrix},  \;\;a = -b, \; c =\frac{1}{2},
 \end{equation}
 where $a\in \mathbb R^{\times}$.
 It is easy to compute that the solution $C_-$ to $d C_- = C_- \eta$ 
 with initial condition
 $C_-(z=0, \lambda) = \id$ is
 $$
 C_- =
 \exp  \left( z D \right). 
 $$
 Let $C_- = F V_+$ be the loop group decomposition of $C_-$
 with $F \in \LISU$ and $V_+ \in \LSLP$, 
 where $F$ takes values in the loop group of $\ISU$, see 
 \cite[Section 5.1]{BRS:Min} 
 for the explicit decomposition using elliptic functions.
 Let $z = x + i y $ be the complex coordinate and $\gamma$ the $2 \pi k$  translation in $y$-direction, 
 that is, $\gamma^* z = z + 2 \pi i k, k \in \mathbb R$. 
 Then $C_-$ changes as $\gamma^* C_- = \exp (2 \pi i kD) 
 C_-$. Since $M=\exp (2 \pi ik D) \in \LISU$, 
 the loop group decomposition for $\gamma^* C_-$
 is computed as 
 $$
 \gamma^* C_- = \left(M F\right)\cdot (\gamma^{*}V_{+}), \;\;
 MF \in \LISU.
 $$
 Let $\hat f^{\lambda}$ be the immersion defined from $F$ 
 via the Sym formula \eqref{eq:symNil}. Then a straightforward computation 
 shows that $\hat f^{\lambda}$ changes 
 by $\gamma$ as follows:
 \begin{equation}\label{eq:monodromy}
 \gamma^* \hat f^{\lambda} = 
 \left(\ad (M) m - 
 X\right)^o -\frac{1}{2}
 \left(\ad(M) (i \lambda \partial_{\lambda} m)+
 [X, \ad(M) m] 
 -Y\right)^{d}, \;\;
 \end{equation}
 where $m$ is the map defined in \eqref{eq:SymMin}, 
 $$
 X =  i \lambda (\partial_{\lambda} M) M^{-1}
 \;\; \mbox{and} \;\; 
 Y =  i \lambda \partial_{\lambda}X = i \lambda \partial_{\lambda}  (i \lambda (\partial_{\lambda} M) M^{-1}).
 $$
 A direct computation shows that 
 $M|_{\lambda =1} = \di (e^{\pi ik }, e^{-\pi ik})$,
 $$
 X|_{\lambda =1} =
 \begin{pmatrix} 
  0 &  2i a(1-e^{2 \pi i k}) \\
  -2i a(1-e^{-2 \pi i k}) &0
 \end{pmatrix}
 $$
 and 
 $$
 Y|_{\lambda =1} = 4 a^2
 \begin{pmatrix} 
 (e^{2 \pi i k}- e^{-2 \pi i k}) -4 \pi i k &  0 \\
 0 &  -(e^{2 \pi i k}- e^{-2 \pi i k}) +4 \pi i k 
 \end{pmatrix}.
 $$
 By the identification \eqref{eq:Nilidenti}, we see that 
 this gives a helicoidal motion along the $x_3$-axis 
 through the point $(-4a, 0, 0)$ with the angle $2 \pi k$
 and the pitch $8 a^2$, see Appendix \ref{sc:isoNil} 
 for the isometry group of $\Nil$.
 Thus the resulting surface $f^{\lambda}|_{\lambda =1}$ is a helicoid.
 It is easy to see that the Abresch-Rosenberg 
 differential of a helicoid
 is $4 B^{\lambda} dz^2 =-4 a^2 \lambda^{-2}dz^2, \; a \in \mathbb R^{\times}$.

 Choosing some appropriate $\LISU$-initial condition for $C_-$
 above will yield a surface of revolution, that is, 
 a catenoid surface. Catenoids and helicoids are not 
 isometric in $\Nil$, even though their Gauss map are isometric  
 in $\mathbb H^2$.
\begin{Remark}
\mbox{}
\begin{enumerate}
\item 
 If the parameter $a$ in the potential \eqref{eq:delpot} 
 is chosen properly, 
 then the resulting surface is the standard helicoid as in \eqref{eq:standardhelicoid}. 
 It is a minimal helicoid in $\mathbb E^3$, see Appendix \ref{subsc:helicoid}.
\item
 The holomorphic potential defined in \eqref{eq:delpot} 
 produces, via the immersion $m$, a nonzero constant mean curvature 
 surface $m$ of revolution in Minkowski $3$-space, \cite[Section 5.1]{BRS:Min}. 
 More precisely, the axis of this surface of revolution is timelike. 
 It would be interesting to know what surfaces correspond to the 
 above potential with the condition $(a+b)^2-c^2$ 
 negative with $a \neq -b$, positive or zero, that is,
 the axis is timelike which is not parallel to $e_3$, 
 spacelike or lightlike in Minkowski $3$-space, 
 respectively.

\item To the best of our knowledge, the associated family of a helicoid gives 
 a new family of minimal surfaces. All these surfaces have the same support function.
\end{enumerate}
\end{Remark}
\appendix
\section{Surfaces with holomorphic Abresch-Rosenberg differential}\label{sc:Appendix}
 \subsection{}
 In this appendix, we determine all surfaces with 
 holomorphic Abresch-Rosenberg differential. 
\begin{Theorem}\label{thm:ARholo}
 Let $f$ be a conformal immersion into $\Nil$ and $B$ its Abresch-Rosenberg differential 
 defined in \eqref{eq:ARdiff}. If $B$ is holomorphic, then the surface $f$ 
 is one of the following{\rm:}
\begin{enumerate}
\item A constant mean curvature surface.\label{en:CMC}
\item A Hopf cylinder.
\end{enumerate}
\begin{proof}
 The structure equations for a surface with holomorphic $B$ can be 
 phrased as
 \begin{eqnarray}
 & -B H_{\bar z}e^{-w/2+u/2}= H_z e^{u/2},\label{a1}\\
 &  \frac{1}{2} w_{z \bar z} + e^{w} -|B|^2 e^{-w}+\frac{1}{2}(H_{z \bar z}
 +p ) e^{-w/2 + u/2}  =0, \label{a2}\\
 & - \bar B H_z e^{-w/2+u/2}=H_{\bar z} e^{u/2},\label{a3} 
 \end{eqnarray}
 where $p$ is $H_z (-w/2 + u/2)_{\bar z}$ or 
 $H_{\bar z} (-w/2 + u/2)_{z}$ respectively. Since $H$ is real, 
 taking the complex conjugate of \eqref{a1} and inserting 
 it into \eqref{a3}, we obtain 
 \begin{equation}\label{eq:Bholocond}
 \bar B H_z e^{-\bar w/2} = \bar B H_z e^{-w/2}.
 \end{equation}
 This equation holds if $B =0$ or $H={\rm const}$ or 
 $e^{-\bar w/2}= e^{- w/2}$. If $H={\rm const}$, then we 
 are in case (\ref{en:CMC}). Let us assume now $H$ not 
 constant.
 If $B$ is identically zero, then \eqref{a1} and \eqref{a3} show 
 that $H$ is constant.
 We assume now $B\neq 0$ and $H \neq {\rm const}$. 
 Then the equation \eqref{eq:Bholocond} implies $e^{w/2}= e^{\bar w/2}$. 
 Using the identity $e^{w/2} =- H e^{u/2}/2 +i h/4$, we obtain that 
 the support function $h = 2 (|\psi_1|^2 -|\psi_2|^2)$ is 
 equal to zero, that is $|\psi_1| = |\psi_2|$. Thus the surface is 
 tangent to $E_3$ by Proposition \ref{heightzero} and this condition is equivalent to that 
 the surface is a Hopf cylinder.
\end{proof}
\end{Theorem}
 General Hopf cylinders are of constant mean curvature $H$ 
 if and only if the curvature of the base curve is constant 
 and equal to $2H$. 
 Therefore the only minimal Hopf cylinders are 
 vertical planes.
 In the proof of the above theorem, 
 we have seen the case where the holomorphic differential 
 $B$ vanishes identically. This describes in fact  
 constant mean curvature surfaces. From \cite{Abresch-Rosenberg}, 
 such surfaces are classified as follows. 
\begin{Proposition}[\cite{Abresch-Rosenberg}]
 The surfaces with identically vanishing Abresch-Rosenberg differential are
 constant mean curvature surfaces and they are classified as follows:
\begin{enumerate}
\item For $H \neq 0$, they are spheres of revolution.
\item For $H =0$, they are 
 vertical planes or horizontal umbrellas. 
\end{enumerate}
\end{Proposition}

\subsection{}\label{subsc:familysubgroup}
 It is known that any two-dimensional Lie subgroup of 
 $\mathrm{Nil}_3$  
 is normal (see for example \cite[Corollary 3.8]{MP}). 
 Moreover all two-dimensional Lie subgroups belong to the 
 $1$-parameter family 
 $\{G(t)\> \vert\ t\in \mathbb{R}P^1\}$ 
 of normal subgroups defined by
\begin{equation}\label{normalsubgroup}
 G(t) = \{(x_1,tx_1,x_3) | x_1,x_3 \in \mathbb{R}\}.
 \end{equation}
 Here the coordinate plane 
 $x_1 = 0$, that is
 $\{(0, x_2, x_3) | x_2,x_3\in \mathbb{R}\}$, 
 is regarded as $G(\infty)$. 
 Note that $G(0)$ is the coordinate plane 
 $x_2 = 0$. 
 For any $t\not=t^\prime$, 
 $G(t)$ and $G(t^\prime)$ only intersect along 
 the subgroup 
 $\varGamma=\{(0,0,x_3)|x_3\in \mathbb{R}\}$, that is, the $x_3$-axis. 
 There are no more two-dimensional Lie subgroups \cite[Theorem 3.6-(5)]{MP}. 
 Each $G(t)$ is realized as a vertical plane in $\mathrm{Nil}_3$. 
 Every vertical plane is congruent to $G(t)$ for some $t\in \mathbb{R}P^1$.

\section{Isometry group of three-dimensional Heisenberg group}\label{sc:isoNil}
 \subsection{}
 The identity component $\mathrm{Iso}_{\circ}(\mathrm{Nil}_{3}(\tau))$
 of the isometry group of $\mathrm{Nil}_3(\tau)$ is the semi-direct product 
 $\mathrm{Nil}_3(\tau)\rtimes \mathrm{U}_1$ if $\tau\not=0$. Here 
 $\mathrm{U}_1$ is identified with $\mathbb{S}^1=\{e^{i\theta}
 \ \vert \ \theta\in \mathbb{R}\}$.

 The action of $\mathrm{Nil}_3(\tau)\rtimes \mathrm{U}_1$ is given by
$$
((a,b,c),e^{i\theta})\cdot (x_1,x_2,x_3)=
(a,b,c)\cdot (\cos \theta{x}_1-\sin\theta{x}_2,\sin\theta{x}_1+\cos\theta{x}_2,x_3).
$$
 The Heisenberg group $\mathrm{Nil}_3$ itself acts on $\mathrm{Nil}_3$ by left translations
 and is represented by $(\mathrm{Nil}_{3}(\tau)\rtimes \mathrm{U}_1)/\mathrm{U}_1$ as 
 a naturally reductive homogeneous space. One can see that 
 this homogeneous space is not Riemannian symmetric. 
\subsection{}\label{subsc:helicoid}
 The Lie algebra $\mathfrak{iso}(\mathrm{Nil}_{3}(\tau))$ is generated 
 by four Killing vector fields $E_1$, $E_2$, $E_3$ and 
 $E_4=-x_{2}\partial_{x_1}+x_{1}\partial_{x_2}$. 
 The commutation relations are:
$$
[E_4,E_1]=E_2,
\
[E_4,E_2]=-E_1,
\
[E_1,E_2]=E_3.
$$
 The $1$-parameter transformation group 
 $\{\rho_\theta\}$ generated by $E_4$ consists of 
 rotations $\rho_\theta=((0,0,0),e^{i\theta})$ of angle $\theta$ along the $x_3$-axis. 
 An isometry $\rho^{(\mu)}_t\in \mathrm{Nil}_{3}(\tau)\rtimes \mathrm{U}_1$ of the form 
 $\rho_{t}^{(\mu)} =((0,0,\mu{t}),e^{it})$ 
 is called a \textit{helicoidal motion with pitch} $\mu$. In particular,  
 a helicoidal motion with pitch $0$ is nothing but a rotation $\rho_t$.
\begin{Definition}
 A conformal immersion $f:M\to \mathrm{Nil}_3(\tau)$ is said to be a 
 \textit{helicoidal surface} if it is invariant under some helicoidal motion. 
 In particular, $f$ is said to be a \textit{surface of revolution} if it is invariant under some 
 rotation $\rho_t$.
\end{Definition}
 The standard helicoid 
\begin{equation}\label{eq:standardhelicoid}
 f(x_1,x_2)=(x_1,x_2,\mu\tan^{-1}(x_2/x_1))
\end{equation}
 is a helicoidal minimal surface in $\mathrm{Nil}_3(\tau)$. 
 In fact this surface is invariant under helicoidal motions 
 of pitch $\mu$. Note that 
 this helicoid is minimal in any $E(\kappa,\tau)$, \cite{BBBI}.

 Tomter \cite{Tomter} studied (non-minimal) constant mean 
 curvature rotational surfaces in 
 $\mathrm{Nil}_3$.  Caddeo, Piu and Ratto \cite{CPR} 
 studied rotational surfaces 
 of constant mean curvature (including minimal surfaces) 
 in $\mathrm{Nil}_3$ by the equivariant submanifold geometry
 (in the sense of W.~Y.~Hsiang).  
 Figueroa, Mercuri and Pedrosa \cite{FMP} investigated 
 surfaces of constant mean curvature 
 invariant under some $1$-parameter isometry groups.
\subsection{}
 In this subsection, we classify homogeneous surfaces in 
 $\mathrm{Nil}_3$.
\begin{Definition}
 A surface $f:M\to \mathrm{Nil}_3$ is said to be \textit{homogeneous} 
 if there exists a connected Lie subgroup $G$ of 
 $\mathrm{Iso}_{\circ}(\mathrm{Nil}_3)$ which acts transitively on the 
 surface. 
\end{Definition}
 Then we have the following classification of homogeneous surfaces in 
 $\mathrm{Nil}_3$, cf., \cite{IKN}.
\begin{Proposition}\label{homogenroussurfaces}
 Homogeneous surfaces in $\mathrm{Nil}_3$ are congruent to one of the 
 following surfaces{\rm:}
\begin{enumerate}
 \item An orbit of a normal subgroup $G(t)$ defined in {\rm(\ref{normalsubgroup})}.
 \item An orbit of the Lie subgroup $\{((0,0,s),e^{it})\ \vert \ s,t\in \mathbb{R}\}
 \subset \mathrm{Nil}_3\rtimes \mathrm{U}_1$. 
\end{enumerate} 
 In the former case, surfaces are vertical planes. Surfaces in the latter case are 
 Hopf cylinders over circles. Thus the only homogeneous minimal surfaces in 
 $\mathrm{Nil}_3$ are vertical planes. 
\end{Proposition}
\begin{proof}
 Let $f : M \to \mathrm{Nil}_3$ be a conformal homogeneous immersion
 with the transitive group $G$.
 We first show that without loss of generality the surface $f(M)$
 contains the identity element $\id$ of $\mathrm{Nil}_3$. Let $z_0 \in M$, 
 then $\hat f = L_{f(z_0)^{-1}} \circ f$ is a conformal homogeneous 
 immersion with group $\hat G = L_{f(z_0)^{-1}} G L_{f(z_0)}$.
 Since $\hat f(z_0) = \id$, the claim follows.

 We next show that every conformal homogeneous immersion $f : M \to 
 \mathrm{Nil}_3$ admits a two-dimensional transitive group.
 If the transitive group $G$ has $\dim G =4$, then 
 $G = \mathrm{Iso}_{\circ}(\mathrm{Nil}_3)$. However then $f (M) = 
 G . \id = \mathrm{Nil}_3$, which contradicts the fact $f(M)$
 has dimension $2$.
 If the transitive group $G$ has $\dim G =3$, then  the isotropy 
 subgroup $G_{\id}$ of $G$ has dimension $1$. Since 
 $\mathrm{Iso}_{\circ}(\mathrm{Nil}_3) = \mathrm{Nil}_3 \rtimes \mathrm{U}_1$ with 
 normal subgroup $\mathrm{Nil}_3$, we can write every $g_o \in G_{\id}$
 in the form $g_o = (n, s)$, where $n \in \mathrm{Nil}_3$ and 
 $s \in \mathrm{U}_1$. Then $\id = g_o . \id = n$. Hence $g_o 
 \in \mathrm{U}_1$.
 Therefore $G_{\id} = \mathrm{U}_1$. On the Lie algebra level, 
 let $n_1 \oplus 
 s_1$ and $n_2 \oplus  s_2$, $0 \oplus 1$ be a basis for 
 $\mathfrak g = \operatorname{Lie} G$. Then also 
 $n_1  \oplus 0$, $n_2 \oplus 0$ and $0 \oplus 1$ are a basis 
 of $\mathfrak g$. Moreover, $n_1 \oplus 0$ and $n_2 \oplus 0$ 
 generate a two-dimensional subalgebra $\mathfrak g_o$. Let 
 $G_o$ denote the corresponding connected subgroup of $G$, then 
 $G = G_o \cdot \mathrm{U}_1$ and $G_o . \id = G . \id = f(M)$. 
 Hence the claim follows.

 Next we consider the family $G(t), \; (t \in \mathbb RP^1),$ of 
 two-dimensional abelian normal subgroups of $\mathrm{Nil}_3$ 
 defined in Section \ref{subsc:familysubgroup}. 
 It is known that the groups $G(t)$ are the only two-dimensional 
 Lie subgroups of $\mathrm{Nil}_3$.

 We now classify the two-dimensional connected transitive 
 Lie subgroups $G_o$ of $\mathrm{Iso}_{\circ}(\mathrm{Nil}_3) =\mathrm{Nil}_3 \rtimes 
 \mathrm{U}_1$.
 Let $a_1 = n_1 \oplus s_1$ and $a_2 = n_2 \oplus s_2$ be a basis of 
 $\mathfrak g_o = \operatorname{Lie} G_o$. If $s_1 = s_2 =0$, then 
 $G_o \subset \mathrm{Nil}_3$ and the claim follows by what was 
 said above. Assume $s_1 \neq 0$. Then after some subtraction and 
 some scaling we can assume $a_1 = n_1 \oplus 1$ and $a_2 = n_2$.
 For the commutator of $a_1$ and $a_2$ we obtain 
\begin{equation}\label{eq:commutator}
[a_1, a_2] = [n_1, n_2] + [1, n_2] \oplus 0.
\end{equation}
 In the following discussion, we choose 
 the basis $1$ of $\mathfrak u_1$ such that the left translation of $1$ 
 to be the Killing vector field $E_4$ defined in \ref{subsc:helicoid}.

 Case $1$, $[a_1, a_2]=0$: We obtain  $[n_1, n_2] + [1,  n_2]=0$.
 Since $[n_1, n_2]$ is a multiple of $e_3$, 
 putting $n_2 =  \alpha_1 e_1 + \alpha_2 e_2 + \alpha_3 e_3$ and 
 $[n_1, n_2] = \beta e_3$, 
 we conclude $\alpha_1 = \alpha_2 =0$ and $n_2 = \alpha_3 e_3$.

 Case $2$, $[a_1, a_2]\neq 0$: In this case we obtain $ 
 \alpha n_2 (= \alpha a_2) =[n_1, n_2] + [1,  n_2]$ with $\alpha \neq 0$, 
 since we do not have the $a_1(=n_1 \oplus 1)$-component. 
  Putting 
 $n_2 = \alpha_1 e_1 + \alpha_2 e_2 + \alpha_3 e_3$, 
 $[n_1, n_2] = \beta e_3$, we obtain
 $$
 \alpha (\alpha_1 e_1 + \alpha_2 e_2 + \alpha_3 e_3) 
 = \beta e_3 + \alpha_1 e_2 - \alpha_2 e_1 + \alpha_3 e_3,
 $$
 whence $\alpha \alpha_1 = - \alpha_2$, $\alpha \alpha_2 = \alpha_1$, 
 $\beta + \alpha_3 = \alpha \alpha_3$. Thus $\alpha_1 = - \alpha^2 \alpha_1$ and 
 $\alpha_1 =0$. Hence also $\alpha_2 =0$ and $n_2 = e_3$ can be assumed. 
 But then $[a_1, a_2]=[a_1, e_3] =0$, which is a contradiction.
 Altogether we have obtained the following list of possible Lie algebras:
 $$ 
 \mathfrak g(t) = \operatorname {Lie} G(t) \;\;\mbox{and}\;\;
 <n_1 \oplus 1, e_3>.
 $$
 This completes the proof.
\end{proof}
\begin{Remark}
 Since homogeneous Riemannian spaces are complete, we find 
 for any $g \in G$, there exists a unique $\gamma \in \operatorname{Aut}
 \widetilde M$ and $z \in \widetilde M$:
 $$
f(\gamma \cdot z) = g. f(z).
 $$
 Here $ \widetilde M$ denotes the universal cover of $M$.
\end{Remark}
\section{Spin structure on Riemann surfaces}\label{sc:spin}
 A spin structure on an oriented Riemannian $n$-manifold $(M,g)$ is 
 a certain principal fiber bundle over $M$ with structure group $\mathrm{Spin}_n$ which is 
 a $2$-fold covering over the orthonormal frame bundle 
 $\mathrm{SO}(M)$ of $M$.
 In the two-dimensional case, a spin structure can be 
 defined in the following manner \cite{KS}. 

 A \textit{spin structure} on a Riemann surface $M$ is a 
 complex line bundle $\varSigma$ over $M$ together with a smooth 
 surjective fiber-preserving map $\mu:\varSigma\to K_M$ to the 
 holomorphic cotangent bundle $K_M$ of $M$ satisfying 
$$
\mu(\alpha {s})=\alpha^{2}\mu(s)
$$
 for any section $s$ of $\varSigma$ and any function $\alpha$. 
 One can see that 
 $\varSigma\otimes \varSigma$ is isomorphic to $K_M$.
 The complex line bundle $\varSigma$ is called the 
 \textit{spin bundle} and the section $s$ of $\varSigma$ 
 is called a \textit{spinor} of $M$. 
 The squaring map $\mu$ is kept implicit by writing $s^2$ for $\mu(s)$ and $st$ for 
 $\{\mu(s+t)-\mu(s-t)\}/4$. Take a local complex coordinate $z$ on $M$. Then
 there exist two sections of $\varSigma$ whose images under 
 $\mu$ are $dz$. Choose one of these sections,
 and refer to it consistently as $(dz)^{1/2}$. 
 Under this notation, every spinor can be expressed locally in the form $\psi (dz)^{1/2}$.

\section{Harmonic maps into reductive homogeneous spaces}
 \label{sc:harmonic-homogeneous}
 \subsection{}
 Let $G/K$ be a reductive homogeneous space. 
 We equip $G/K$ with a $G$-invariant Riemannian metric which is
 derived from a left-invariant Riemannian metric on $G$.
 
 Then the orthogonal complement $\mathfrak{p}$ of 
 the Lie algebra $\mathfrak{k}$ of $K$ 
 can be identified with the tangent space of $G/K$ 
 at the origin $o=K$.  
 The Lie algebra $\mathfrak{g}$ is decomposed into the orthogonal direct sum:
$$
\mathfrak{g}=\mathfrak{k}\oplus \mathfrak{p}
$$
 of linear subspaces. We define a symmetric bilinear map
 $U:\mathfrak{p}\times \mathfrak{p}\to \mathfrak{p}$ by
$$
 2\langle U(X,Y),Z\rangle=\langle X,[Z,Y]_{\mathfrak{p}}\rangle+
 \langle Y,[Z,X]_{\mathfrak{p}}\rangle,\ \ 
 X,Y,Z \in \mathfrak{p},
$$
 where $[Z,Y]_{\mathfrak{p}}$ denotes the $\mathfrak{p}$-component of 
 $[Z,Y]$. A Riemannian reductive homogeneous space $G/K$ is said to be 
 \textit{naturally reductive} if $U=0$. In particular, when $G$ is a compact semi-simple  
 Lie group and the $G$-invariant Riemannian metric on $G/K$ is derived from a bi-invariant 
 Riemannian metric of $G$, then $G/K$ is said to be a 
 \textit{normal Riemannian homogeneous space}.
 Normal Riemannian homogeneous spaces are naturally reductive.
 Note that in case $K =\{ \id \}, G/K=G$, $U$ is related to 
 the symmetric bilinear map $\{\cdot,\cdot\}$ 
 defined in (\ref{anticommutator}) 
 by $2U=\{\cdot,\cdot\}$.

\subsection{}
 A smooth map $f:M \to N$ of a Riemannian $2$-manifold $M$ 
 into a Riemannian 
 manifold $N$ is said to be a \textit{harmonic map} if it is 
 a critical point of the energy 
$$
E(f)=\int \frac{1}{2}|df|^2\>dA
$$
 with respect to all compactly supported variations.
 It is well known that a map $f$ is harmonic if and 
 only if its tension field  $\mathrm{tr}(\nabla df)$ vanishes. 
 The harmonicity is invariant under conformal transformations of $M$.
 When the target space $N$ is a Riemannian reductive homogeneous space $G/K$,
 the harmonic map equation for $f$ has a particularly simple form.
 The harmonic map equation for maps into Lie groups was already 
 discussed in Section \ref{sc:Preliminaries}.
 Therefore we assume now that $\dim K\geq 1$.

 Let $f:\mathbb{D} \to G/K$ be a smooth map from a simply connected
 domain $\mathbb D \subset \C $ into a Riemannian reductive 
 homogeneous space. 
 Take a frame $F:\mathbb{D} \to G$ of $f$ and put $\alpha:=F^{-1}dF$.
 Then we have the identity (\textit{Maurer-Cartan equation}):
$$
 d\alpha+\frac{1}{2}[\alpha \wedge \alpha]=0.
$$
 Decompose $\alpha$ along the Lie algebra decomposition
 $\mathfrak{g}=\mathfrak{k}\oplus\mathfrak{p}$  in the form
$$
\alpha=\alpha_{\mathfrak{k}}+\alpha_{\mathfrak{p}},
\ \
\alpha_{\mathfrak{k}}\in
\mathfrak{k},
\ \
\alpha_{\mathfrak{p}}
\in
\mathfrak{p}.
$$
 We decompose $\alpha_{\mathfrak{p}}$ with respect to the conformal structure of $\mathbb{D}$ as
$$
\alpha_{\mathfrak p}=
\alpha_{\mathfrak p}^{\prime}+
\alpha_{\mathfrak p}^{\prime \prime},
$$
 where $\alpha_{\mathfrak{p}}^{\prime}$ and 
 $\alpha_{\mathfrak p}^{\prime \prime}$
 are the $(1,0)$ and $(0,1)$ part of 
 $\alpha_{\mathfrak p}$, respectively.

 The harmonicity of $f$ is equivalent to
\begin{equation}\label{harmonicity}
d(*\alpha_{\mathfrak p})+
[\alpha_{\mathfrak{k}} \wedge *\alpha_{\mathfrak p}]
=U(\alpha_{\mathfrak{p}}
\wedge
*\alpha_{\mathfrak{p}}
),
\end{equation}
 where $*$ denotes the Hodge star operator of $\mathbb{D}$.
 The Maurer-Cartan equation is split into its $\mathfrak{k}$-component and 
 $\mathfrak{p}$-component:
\begin{equation}\label{MC-h}
d\alpha_{\mathfrak{k}}+\frac{1}{2}
[\alpha_{\mathfrak{k}}\wedge \alpha_{\mathfrak{k}}]
+[\alpha_{\mathfrak{p}}^{\prime}\wedge 
\alpha_{\mathfrak{p}}^{\prime\prime}]_{\mathfrak{k}}=0,
\end{equation}
\begin{equation}\label{MC-p}
d\alpha
_{\mathfrak{p}}^{\prime}
+
[
\alpha_{\mathfrak{k}}
\wedge 
\alpha_{\mathfrak{p}}^{\prime}
]
+
d\alpha_{\mathfrak{p}}^{\prime\prime}
+
[
\alpha_{\mathfrak{k}}
\wedge 
\alpha_{\mathfrak{p}}^{\prime\prime}
]
+
[
\alpha_{\mathfrak{p}}^{\prime}
\wedge 
\alpha_{\mathfrak{p}}^{\prime\prime}
]_{\mathfrak{p}}
=0.
\end{equation}
 Hence for a harmonic map $f:\mathbb{D}\to G/K$ with a framing $F$, 
 the pull-back $1$-form $\alpha=F^{-1}dF$ satisfies
 (\ref{harmonicity}), (\ref{MC-h}) and (\ref{MC-p}). 
 Combining (\ref{harmonicity}) and (\ref{MC-p}), we obtain
\begin{equation}\label{harm+MC}
d\alpha^{\prime}_{\mathfrak p}+
[\alpha_{\mathfrak k} \wedge 
\alpha_{\mathfrak p}^{\prime}]=
-\frac{1}{2}[\alpha_{\mathfrak p}^{\prime}
\wedge 
\alpha_{\mathfrak p}^{\prime \prime}
]_{\mathfrak p}+
U(\alpha_{\mathfrak p}^{\prime}
\wedge \alpha_{\mathfrak p}^{\prime \prime}).
\end{equation}
 One can easily check that the harmonic map equation for $f$ combined with
 the Maurer-Cartan equation is equivalent to the system (\ref{MC-h}) and (\ref{harm+MC}).

 Assume that
\begin{equation}\label{admissible}
[\alpha_{\mathfrak p}^{\prime}
\wedge 
\alpha_{\mathfrak p}^{\prime \prime}
]_{\mathfrak p}=0, \ \ 
U(\alpha_{\mathfrak p}^{\prime}
\wedge \alpha_{\mathfrak p}^{\prime \prime})=0.
\end{equation}
 Then the harmonic map equation together with the Maurer-Cartan equation
 is reduced to the system of equations:
$$
d\alpha_{\mathfrak p}^{\prime}+
[\alpha_{\mathfrak k} \wedge 
\alpha_{\mathfrak p}^{\prime}]=0, \;\;
d\alpha_{\mathfrak k}+
\frac{1}{2}
[\alpha_{\mathfrak k} \wedge
\alpha_{\mathfrak k}]
+[\alpha_{\mathfrak p}^{\prime}
\wedge
\alpha_{\mathfrak p}^{\prime \prime}
]=0.
$$
 This system of equations is equivalent to the
 following \textit{zero-curvature representation}:
$$
d\alpha^{\lambda}+
\frac{1}{2}[\alpha^{\lambda}
\wedge \alpha^{\lambda}]
=0,
$$
where $\alpha^{\lambda}:=
\alpha_{\mathfrak h}
+\lambda^{-1}\alpha_{\mathfrak p}^{\prime}+
\lambda\> \alpha_{\mathfrak p}^{\prime \prime}$
with $\lambda \in\mathbb{S}^1$. 

\begin{Proposition}
 Let $\mathbb D$ be a domain in  $\C$ and 
 $f:\mathbb{D} \to G/K$ a harmonic map
 which satisfies the admissibility condition {\rm (\ref{admissible})}.
 Then the loop of connections $d+\alpha^{\lambda}$ is flat for
 all $\lambda$. Namely{\rm:}
\begin{equation}\label{eq:flatconnections}
d\alpha^{\lambda}+\frac{1}{2}
[\alpha^{\lambda}\wedge
\alpha^{\lambda}]=0
\end{equation}
 for all $\lambda$.
 Conversely assume that $\mathbb{D}$ is simply connected. Let
 $\alpha^{\lambda}=\alpha_{\mathfrak k}+\lambda^{-1} \alpha_{\mathfrak
 p}^{\prime}+\lambda\alpha_{\mathfrak  p}^{\prime\prime}$ be an
 $\mathbb{S}^1$-family of $\mathfrak{g}$-valued $1$-forms 
 satisfying \eqref{eq:flatconnections} for all $\lambda \in \mathbb{S}^1$.
 Then there exists a $1$-parameter family of maps
$F^{\lambda} :\mathbb{D} \to G$ such that
$$
(F^{\lambda})^{-1}dF^{\lambda}=\alpha^{\lambda}
\;\;\;\mbox{and}\;\;\;
f^{\lambda}=F^{\lambda}\;{\mbox{\rm mod}}\; K:\mathbb{D} \to G/K
$$
 is harmonic for all $\lambda$.
\end{Proposition}
 The $1$-parameter family $\{f^\lambda\}_{\lambda \in \mathbb{S}^1}$ 
 of harmonic maps 
 is called the \textit{associated family} of the original 
 harmonic map $f=f^{\lambda}\vert_{\lambda=1}$ which satisfies the 
 admissibility condition. The map $F^\lambda$ is called an  
 \textit{extended frame} of $f$. 
 When the target space $G/K$ is a Riemannian symmetric space 
 with a semi-simple $G$, then the
 admissibility condition is fulfilled automatically for any $f$, 
 since $U=0$ and 
 $[\mathfrak{p},\mathfrak{p}]\subset \mathfrak{k}$. In the case 
 $G/K=\mathrm{Nil}_{3}\rtimes\mathrm{U}_1/\mathrm{U}_1$,  
 harmonic maps into $\mathrm{Nil}_3$ do  in general not satisfy 
 the admissibility condition.  
 Note that harmonic maps into a naturally 
 reductive Riemannian homogeneous space $G/K$ 
 satisfying the admissibility condition 
 are called \textit{strongly harmonic maps} in \cite{Khemar}. 
 Note that all the  examples of minimal surfaces in 
 $\mathrm{Nil}_3\rtimes \mathrm{U}_1/\mathrm{U}_1$
 discussed in this paper do not satisfy the admissibility condition. 
%
%
%
\bibliographystyle{plain}
\def\cprime{$'$}

\end{document}